\documentclass[12pt]{article}
\usepackage[numbers,sort&compress]{natbib}
\usepackage{enumerate}
\usepackage{amscd}
\usepackage{amsmath}
\usepackage{latexsym}
\usepackage{amsfonts}
\usepackage{setspace}
\usepackage{amssymb}
\usepackage{amsthm}
\usepackage{verbatim}
\usepackage{mathrsfs}
\usepackage{enumerate}
\usepackage[hypertexnames=false]{hyperref}

\theoremstyle{definition}\newtheorem{definition}{Definition}[section]
\theoremstyle{definition}\newtheorem{theorem}{Theorem}[section]
\theoremstyle{plain}\newtheorem{lemma}[theorem]{Lemma}
\theoremstyle{plain}
\theoremstyle{plain}\newtheorem{prop}[theorem]{Proposition}
\theoremstyle{definition}\newtheorem{remark}{Remark}[section]
\usepackage{xcolor}

\newcommand{\e}{\mathrm{e}}

\newcommand{\dd}{\mathrm{~d}}

\newcommand{\mr}{\mathbb{R}}

\newcommand{\Bp}{B^{\frac{2}{p}+1}_{p,1}}
\newcommand{\lef}{\left\|}
\newcommand{\rig}{\right\|}
\newcommand{\define}{\stackrel{\mathrm{def}}{=}}

\allowdisplaybreaks

\numberwithin{equation}{section}
\begin{document}
	\title{Mild ill-posedness in $W^{1,\infty}$ for the  incompressible porous media equation}

\author{Yaowei Xie, Huan Yu}

\author{Yaowei Xie\footnote{School of Mathematical Sciences, Capital Normal University, Beijing 100048, PR China. Email: mathxyw@163.com},~~~\,\,\,\,Huan Yu\footnote{School of Applied Science, Beijing Information Science and Technology University, Beijing, 100192, P.R.China.  Email: huanyu@bistu.edu.cn}}

\date{}
\maketitle

\begin{abstract}
	In this paper, we establish  the mild ill-posedness of 2D IPM equation in the critical Sobolev space $W^{1,\infty}$ when the initial data are small perturbations of  stable profile $g(x_2).$ Consequently, instability can be inferred.  Notably, our results are valid for arbitrary vertically stratified density profiles $g(x_2)$ without imposing any restrictions on the sign of  $g'(x_2).$ From a physical perspective, since gravity acts downward, density profiles satisfying $g'(x_2) < 0$ typically correspond to stable configurations, whereas those with 	$g '(x_2) > 0$ are generally expected to be unstable.Surprisingly, our analysis uncovers an unexpected instability even when $g'(x_2) < 0$ and $g'(x_2)\in W^{2,\infty}(\mr)$.
 To the best of our knowledge, this work provides the first rigorous demonstration of IPM instability for vertically nonlinear density profiles, marking a significant departure from conventional physical expectations.

\end{abstract}
\section{Introduction}

In this paper, we consider the 2D incompressible porous medium (IPM) system, which consists of an active scalar equation with a velocity field $u({\bf x},t)$  satisfying the momentum equation given by Darcy's law. That is,
\begin{align}\label{1.1}
	\begin{cases}
		\partial_t \rho+\left(u\cdot \nabla\right)\rho=0,\\
		u=-\nabla p-(0,\rho),\\
		\nabla\cdot u=0,\\
		\rho({\bf x},0)=\rho_0,
	\end{cases}
\end{align}
where $({\bf x}, t) \in\mr^2\times \mr^+$ with ${\bf x}=(x_1,x_2)$, $\rho({\bf x},t)$  represents the  density transported  by the fluid, $u({\bf x},t)$ is  the incompressible velocity, and $p({\bf x},t)$ is the  pressure. We refer to
\cite{ipm-phy,castro-2019-local-arma} and references therein for further explanations on the physical background and applications of this model.

By utilizing the incompressibility condition, the Biot-Savart law for the velocity field of 2D IPM equation can be  expressed as either  $$u =\nabla^\perp(-\Delta)^{-1}\partial_{x_1}\rho~~ \text{or} ~~u=\nabla^\perp (-\Delta)^{-\frac12}\mathcal{R}_1\rho,$$ where $\mathcal{R}_1$ represents the first component of the Riesz transform. As a result, the regularity of the velocity in system \eqref{1.1} is the same as that of the 2D  SQG equation but is one order lower than that of the 2D Euler equation.  Additionally, a notable distinction is that the Biot-Savart law of the 2D IPM equation contains  a horizontal partial derivatives $\partial_{x_1}$. This feature leads to the existence of relatively simple steady-state solutions of the form $\rho_s(x)=g(x_2)$. 
 
Let us denote $\eta({\bf x},t)\define \rho({\bf x},t)-g(x_2),$ where $\rho({\bf x},t)$ is the solution of system \eqref{1.1}, then $\eta$ satisfies the following perturbation equation:  
\begin{align}\label{p-IPM}
	\begin{cases}
		\partial_t \eta+\left(u\cdot \nabla\right)\eta=-g'(x_2)\mathcal{R}_1^2 \eta,\\[1mm]
		u=\nabla^\perp (-\Delta)^{-\frac12}\mathcal{R}_1\eta,\\[1mm]
		\eta({\bf x},0)=\eta_0({\bf x})=\rho_0({\bf x})-g(x_2).
	\end{cases}
\end{align}
Here, $\mathcal{R}_1^2$  on the right-hand side of the perturbation equation is a negative operator. If $g'(x_2)<0$ (which physically corresponds to a fluid density that varies only in the vertical direction, with the density of each upper layer being less than that of the lower layer, indicating a stably stratified fluid), then the stability of the perturbation equation can be expected.

{\bf Well-posedness}
The 2D IPM equation is locally well-posed in $H^s(X)(s>2)$,  where $X$ can be the whole space $\mr^2$, a periodic domain $\mathbb{T}^2$, or a bounded strip domain $\mathbb{S}=\mathbb{T}\times [-l,l],$  $0<l<\infty$, as seen in references  \cite{castro-2019-local-arma}, \cite{xue-2009-besov-} and \cite{yao-2023-arma-smallscale-ipm}. However, the global well-posedness remains a challenging and unresolved problem, similarly to the SQG equation case.

In recent years, significant progress has been made regarding the global well-posedness of system \eqref{1.1} for
initial data close to  linear stratified density $g(x_2) = -x_2$ in certain Sobolev spaces, which also implies the asymptotic stability of system \eqref{p-IPM}, see \cite{castro-2019-local-arma, elgindi-ipm-2017-arma,paicu-2024-arma-IPM,park-arxiv-2024-IPM-strip,kim-IPM-2024-JFA}.
Elgindi \cite{elgindi-ipm-2017-arma} established the  asymptotic stability of system \eqref{p-IPM}  for initial perturbations in ${H^k}(\mr^2)\cap W^{4,1}(\mr^2)$ and ${H^k}(\mathbb{T}^2)$ with $k \geq 20$. 
Recently,  Bianchini, Crin-Barat and Paicu improved this result  in \cite{paicu-2024-arma-IPM} by assuming  the initial perturbations belonging to
$\dot{H}^{s}(\mr^2)\cap \dot{H}^{\tau}(\mr^2)$ with $0<\tau<1, s\geq 3+\tau$.
For the case of the flat strip $\mathbb{S}=\mathbb{T}\times [-l,l]$ with $0<l<\infty$, the stability of the stratified steady state solutions was proved in \cite{castro-2019-local-arma} under the condition $\|\rho_0-g(x_2)\|_{H^k(\mathbb{S})}\leq \varepsilon$ with $k \geq 10$. More   recently, this result was refined in \cite{park-arxiv-2024-IPM-strip} under slightly less restrictive regularity assumptions, requiring only $\|\rho_0-g(x_2)\|_{H^k(\mathbb{S})}\leq \varepsilon$ with $k \geq 3$. Additionally, we highlight the work in \cite{kim-IPM-2024-JFA}, which establishes asymptotic stability  in the torus $\mathbb{T}^2$, or the strip $\mathbb{S}$, or the whole space $\mathbb{R}^2$ for a perturbation in $H^k$ for $k>3$, provided  that the vertical derivative of the steady
state $g(x_2)$ is sufficiently large, depending on the size of the perturbation. It is also noteworthy that Kiselev and Yao \cite{yao-2023-arma-smallscale-ipm} constructed examples of solutions exhibiting infinite growth of derivatives over time, demonstrating nonlinear instability for a class of stratified steady states of the IPM equation on $\mathbb{T}^2$ and $\mathbb{S}$.

{\bf Ill-posedness} 
Since the pioneering work by Bourgain and Li on the strong ill-posedness of the Euler equations in critical spaces (\cite{euler-2-illpose-bourgain-2015-IA}),  there have been many remarkable results regarding   ill-posedness of the Euler equations (see \cite{euler-1-illpose-elgindi-2017,elgindi2020infty,euler-3-illpose-bourgain-2015-GFA,euler-10-illpose-lxyt-2024-arxiv,euler-11-illpose-my-2016-ma,euler-4-illpose-bourgain-2021-IMRN}) and SQG equations (see \cite{sqg-1-illpose-cjk-2024-arixv,sqg-2-illpose-cordoba-2022-advance,sqg-3-illpose-cordoba-2024-cmp,sqg-4-illpose-kim-jeong-2024-analpde}). 
Given the similarities between the SQG and IPM equations, particularly the same order of their Biot-Savart laws, it is natural to conjecture that analogous ill-posedness results may hold for the IPM equations. However, the notable differences in symmetry and anisotropy between the two systems make the ill-posedness problem of the 2D IPM equations a non-trivial challenge. Recently, Bianchini et al. \cite{cordoba-2024-arxiv-ipm-ill}  posted an article on the arXiv with a significant  strong ill-posedness result in critical space $H^2$  near the linear stratified steady state $g(x_2)=-x_2$. 

   The goal of this paper is to establish the  ill-posedness result for perturbations of  a more general stable profile $g(x_2)$. 
   Specifically, we construct solutions that are initially in the space of compactly supported, infinitely differentiable functions $C_c^\infty(\mr^2)$ and small in the critical Sobolev space $W^{1,\infty}(\mr^2)$, but the solutions   become large in $W^{1,\infty}(\mr^2)$ within a short period of time.
Our main result can be summarized stated as follows.
\begin{theorem}\label{thm1}
	Let $g(x_2)$ be any horizontal stratified state satisfying $g'(x_2)\not\equiv0$ and $g'(x_2)\in W^{2,\infty}(\mr)$. For every sufficiently small $\varepsilon>0$, there exists a initial data $\eta_0({\bf x})\in  C^\infty_c(\mr^2)$ satisfying 
	\begin{align*}
		\left\| \eta_0\right\|_{W^{1,\infty}(\mr^2)}\leq \varepsilon,
	\end{align*}
	such that  the unique local solution $\eta(\cdot,t)$  of system \eqref{p-IPM} with initial data $\eta_0$  satisfies
	\begin{align*}
		\sup_{0\leq t\leq \tilde{c}\varepsilon}\left\| \eta(\cdot,t)\right\|_{W^{1,\infty}(\mr^2)}\geq \tilde{C}.
	\end{align*}
	Here, $\tilde{c}, \tilde{C}$ are universal constants that are independent of $\varepsilon$ but depend on the choice of $g(x_2)$.
\end{theorem}

\begin{remark}In contrast to  the work by Bianchini et al. \cite{cordoba-2024-arxiv-ipm-ill}, 
	  which established the  strong ill-posedness of the IPM equation for initial data that are small $H^2$ perturbations of the linearly stable profile $-x_2$, our work  demonstrates mild ill-posedness in the critical space	$W^{1,\infty}$ (For detailed definitions of strong and mild ill-posedness, readers are referred to  \cite{euler-2-illpose-bourgain-2015-IA} and \cite{elgindi2020infty}). 
\end{remark}
\begin{remark}
	We emphasize that our result holds for \emph{arbitrary} vertically stratified density profiles $g(x_2)$, with \emph{no} assumptions required on the sign of $g'(x_2)$. Physically, since gravity acts downward, density profiles with $g'(x_2) < 0$ typically represent stable configurations (density decreasing with height), while those with $g'(x_2) > 0$ normally correspond to unstable scenarios. 
	Remarkably, our analysis uncovers a surprising instability mechanism that persists even when $g'(x_2) < 0$ and $g'\in W^{2,\infty}(\mathbb{R})$. To the best of  our knowledge, this constitutes the first rigorous demonstration of ill-posedness for the IPM equation with nonlinear vertical density stratification.
\end{remark}


	
	


{\bf Ideas of proof}~~
The main idea demonstrated in this paper can be summarized in the following three steps:

{\bf Step 1. Provide two local prior bounds} 

Assume that $\eta \in L^\infty\left(0,T; B^{\frac{2}{p}+1}_{p,1}(\mr^2)\right)$ for some $T>0$ is a  solution of system  \eqref{p-IPM} corresponding to the initial value $\eta_0\in B^{\frac{2}{p}+1}_{p,1}(\mr^2)$ with $2< p<\infty$, then, for $t\in[0,T],$
\begin{align*}
	\lef \eta(t)\rig_{B^{\frac{2}{p}+1}_{p,1}}\leq \dfrac{\frac{c_g}{2}+2\left\| \eta_0\right\|_{B^{\frac{2}{p}+1}_{p,1}}}{1-C\left(\frac{c_g}{2}+2\left\| \eta_0\right\|_{B^{\frac{2}{p}+1}_{p,1}}\right)t},
\end{align*}
and \begin{align*}
	\|\eta(t)\|_{B^{\frac{2}{p}}_{p,1}}\leq \|\eta_0\|_{B^{\frac{2}{p}}_{p,1}}\e^{Ct\left(\displaystyle\sup_{0\leq s\leq t}\|\eta(s)\|_{B^{\frac{2}{p}+1}_{p,1}}+c_g\right)},
\end{align*}
where  $c_g$  is a fixed constant characterizing the background stratification.

 {\bf Step 2. Construction of the right initial data}

Inspired by the work of Elgindi and Masmoudi on the logarithmic singularity of Riesz operators \cite{elgindi2020infty}, we introduce a novel weighted estimation method for general steady states (where 
$g'(x_2)$ is not identically zero). This approach preserves the delicate balance required for singularity analysis while effectively capturing nontrivial stratification effects.

To establish a lower bound on $\|\nabla \eta\|_{L^\infty},$ it is essential to construct initial data that not only induces growth at
 $t=0,$ but also ensures that the nonlinear evolution does not suppress this growth mechanism for $t>0.$ Specifically, we design the initial data for system \eqref{p-IPM} as follows:
\begin{align*}
	\eta_0({\bf x})=	
		\rho_0({\bf x})-g(x_2)=\dfrac{f_N({\bf x})}{\sqrt{N}},	
\end{align*}
where \begin{align*}
	f_N({\bf x}) =\partial_{x_1}\Delta \big(G_N({\bf x}-{\bf x_0})\chi ({\bf x}-{\bf x_0})\big),
\end{align*}\begin{align*}
	G_N({\bf x})=\left(x_1^3x_2-x_1x_2^3\right)\log(x_1^2+x_2^2+\frac{1}{2^N}).
\end{align*}
Here, the point ${\bf x_0}=(x_0,y_0)$ is selected  to satisfy  $g'({\bf x_0})=g'(y_0)\neq  0$, and  $\chi({\bf x})$ is a smooth cut-off function 
\begin{align*}
	\chi({\bf x})=&
	\begin{cases}
		1, \text{~on~} B({\bf 0}; 1)=\{{\bf x}=(x_1,x_2);~|{\bf x}-{\bf 0}|\leq 1\},&\\
	 0, \text{~on~} B({\bf 0}; 2)^c=\{{\bf x}=(x_1,x_2);~|{\bf x}-{\bf 0}|> 2\}.\end{cases}
\end{align*}


After careful verification,
for $\forall$ $i,j,k,l\in\{1, 2\}$, we conclude that from a degree point of view,
\begin{align*}
	&\partial_i\Delta G_N(x_1,x_2)\approx \left(x_1^2+x_2^2+\frac{1}{2^N}\right)^\frac{1}{2},\\
	&\partial_i\partial_j\Delta G_N(x_1,x_2)\approx 1,\\
	&\partial_i\partial_j\partial_k\Delta G_N(x_1,x_2)\approx \dfrac{1}{\left(x_1^2+x_2^2+\frac{1}{2^N}\right)^\frac{1}{2}},\\
	&\partial_i\partial_j\partial_k\partial_l\Delta G_N(x_1,x_2)\approx \dfrac{1}{x_1^2+x_2^2+\frac{1}{2^N}},
\end{align*} where $\partial_i=\partial_{x_i}.$
However, if $\Delta$ is not applied to $G_N$, and $\partial_{1}^2$ is applied, then logarithmic singularity will appear. In fact, $\partial_{1}^2$ is equal to $-\mathcal{R}_1^2 \Delta$, so we have
\begin{align*}
	\partial_{1}^3	\partial_{2}G_N(x_1,x_2)=6\log(x_1^2+x_2^2+\dfrac{1}{2^N})+G^1_N(x_1,x_2),
\end{align*}
where $\left\|G^1_N\right\|_{L^\infty (B({\bf 0};2)) }\leq C_2$. 

Eventually, it yields that
\begin{align*}
	\|f_N(\cdot)\|_{W^{1,\infty}}\leq c,&~\|f_N(\cdot)\|_{B^{\frac{2}{p}}_{p,1}}\leq  c\left(\log N\right)^{\frac{1}{2}}, ~\|f_N(\cdot)\|_{B^{\frac{2}{p}+1}_{p,1}}\leq cN,\\
	&\|g'(x_2)\mathcal{R}_1^2\nabla^\perp f_N({\bf x})\|_{L^\infty(\mr^2)}\geq c_1N,
\end{align*}
with $2<p<\infty$ and $c, c_1$ are independent of $x_0, y_0$ and $N$. 
More details refer to Lemma \ref{fN-sequence}.
Accordingly, \begin{align*}
	\|\eta_0(\cdot)\|_{W^{1,\infty}}\leq \dfrac{c}{\sqrt{N}},&~\|\eta_0(\cdot)\|_{B^{\frac{2}{p}}_{p,1}}\leq  \dfrac{c\left(\log N\right)^{\frac{1}{2}}}{\sqrt{N}}, ~\|\eta_0(\cdot)\|_{B^{\frac{2}{p}+1}_{p,1}}\leq c\sqrt{N},\\
	&\|g'(x_2)\mathcal{R}_1^2\nabla^\perp \eta_0({\bf x})\|_{L^\infty(\mr^2)}\geq c_1\sqrt{N}.
\end{align*}

{\bf Step 3. Lower bound estimation of $\|\nabla \eta\|_{L^\infty}$.}

To obtain the lower bound estimate for $\|\nabla \eta\|_{L^\infty}$, we apply $\nabla^\perp$ to both sides of \eqref{p-IPM}$_1$ instead of $\nabla.$ This is mainly due to the fact that the Biot-Savart law governing the velocity field $u$ involves the gradient $\nabla^\perp$, applying $\nabla^\perp$  makes the equation more concise.

As a result, 
denote  $\Theta=\nabla^\perp \eta$, then $\Theta$ satisfies \begin{align}\label{theta}
	\begin{cases}
			\partial_t \Theta+u\cdot \nabla\Theta=-g'(x_2)\mathcal{R}_{1}^2 \Theta-\left(\nabla^\perp g'(x_2)\right)\mathcal{R}_{1}^2 \eta+\Theta\cdot\nabla u,\\[1mm]
		u=(-\Delta)^{-\frac12}\mathcal{R}_1\Theta,\\[1mm]
		\Theta({\bf x},0)=\Theta_0=\nabla^\perp\eta_0({\bf x}).
	\end{cases}
\end{align}
Compounding on both sides of $\eqref{theta}_1$ the Lagrangian flow-map $\Phi$  transported by the velocity $u$, we get  
\begin{align*}
			\partial_t \left(\Theta\circ \Phi\right)&+g'(x_2)\mathcal{R}_{1}^2 \left(\Theta\circ\Phi\right)\\
&=\left(-g' \circ\Phi\right)\left[\mathcal{R}_{1}^2 ,\Phi\right]\Theta+\left(g'-g' \circ\Phi\right)\mathcal{R}_{1}^2 \left(\Theta\circ\Phi\right)\\
&~~~~~~+\left(\Theta\cdot\nabla u-\left(\nabla^\perp g' \right)\mathcal{R}_{1}^2 \eta\right)\circ \Phi
\end{align*}	where
\begin{align*}
\left[\mathcal{R}_{1}^2 ,\Phi\right]\Theta=\mathcal{R}_{1}^2 (\Theta)\circ\Phi-\mathcal{R}_{1}^2 \left(\Theta\circ\Phi\right).
\end{align*}
By Duhamel's principle,
\begin{align*}
	\Theta\circ \Phi=\e^{-g'\mathcal{R}_{1}^2 t}\Theta_0&+\int_{0}^t\e^{-g'\mathcal{R}_{1}^2 (t-\tau)}\left(-g' \circ\Phi\right)\left[\mathcal{R}_{1}^2 ,\Phi\right]\Theta(\tau)\dd\tau\\\nonumber
&+\int_{0}^t\e^{-g'\mathcal{R}_{1}^2 (t-\tau)}\left(\left(g'-g' \circ\Phi\right)\mathcal{R}_{1}^2 \left(\Theta\circ\Phi\right)\right)(\tau)\dd\tau\\
&+\int_{0}^t\e^{-g'\mathcal{R}_{1}^2 (t-\tau)}\left(\left(\Theta\cdot\nabla u-\left(\nabla^\perp g' \right)\mathcal{R}_{1}^2 \eta\right)\circ \Phi\right)(\tau)\dd\tau.
\end{align*}
In fact,  we  can obtain the following lower bound estimate (see Section 4.1  for more details)
\begin{align*}
	\|\Theta\|_{L^\infty}\geq &t\left\| -g'\mathcal{R}_{1}^2 \Theta_0\right\|_{L^\infty}-\left\|\Theta_0\right\|_{L^\infty}-Ct^2\left\| \Theta_0\right\|_{B^{\frac{2}{p}}_{p,1}}\\\nonumber
	&~~-Ct^2\|\nabla u\|_{L^\infty_tL^\infty}\left(1+\e^{2t \|\nabla u\|_{L^\infty_tL^\infty}}\right)\|\Theta\|_{L^\infty_tB^{\frac{2}{p}}_{p,1}}\\
	&~~-Ct\left(1+\e^{2t \|\nabla u\|_{L^\infty_tL^\infty}}\right)\left(\|\Theta\|_{L^\infty_tL^\infty}\|\Theta\|_{L^\infty_tB^\frac{2}{p}_{p,1}}+\|\eta\|_{L^\infty_tB^\frac{2}{p}_{p,1}}\right).
\end{align*}
Here, $L^\infty_tL^\infty$ and $L^\infty_tB^\frac{2}{p}_{p,1}$ represent $L^\infty\left([0,t]; L^\infty(\mr^2)\right)$ and $L^\infty\left([0,t]; B^\frac{2}{p}_{p,1}(\mr^2)\right)$, respectively.

Recall the special $\eta_0$ constructed in Step 2 and the prior estimates in Step 1, and we are able to get
\begin{align*}
	\|\Theta\|_{L^\infty}\geq \frac{\bar{c}_1}{2}t\sqrt{N}-\bar{c}_2t^2N-\bar{c}_3t\sqrt{N}\|\Theta\|_{L^\infty_tL^\infty},
\end{align*}
by choosing $t$ sufficiently small and $N$ sufficiently large.
It follows by the continuity method that 
for sufficiently small $t\sim \frac{1}{\sqrt{N}}$, 
\begin{align*}
	\|\Theta(t)\|_{L^\infty}\geq \tilde{C}.
\end{align*}
See
a detailed explanation in Section \ref{sec:5.2}.

The paper is organized as follows. Section \ref{sec:2} provides a list of
lemmas to be used in the proofs of the theorem stated above and the propositions below.  Section \ref{sec:3} presents a detailed construction of a sequence of functions with special properties. The
construction of the initial data in the proof of Theorem \ref{thm1} makes use of this special sequence.
In Section \ref{sec:4}, we briefly describe the local existence and uniqueness of the IPM equation  and  establish two key local priori bounds.  Section \ref{sec:5} is  divided into two parts. Subsection \ref{sec:5.1} is devoted to 
 obtaining a lower bound estimate for $\nabla \eta$ and  subsection \ref{sec:5.2} is devoted to proving Theorem \ref{thm1}.

{\bf Notations:} 
\begin{itemize}
	\item In the following, we  use $C$ to denote various constants whose values may be different from line to line. Occasionally, we use $c$ (or $c$ with a subindex) to mark some crucial constants. 
	\item For notational simplicity, throughout this paper we adopt the following convention unless otherwise specified:
	\begin{enumerate}
		\item In Section \ref{sec:2} (Preliminaries) and the Appendix: $\|f\|_{X}$ denotes $\|f\|_{X(\mr^n)}$,
			\item In all other Sections: $\|f\|_{X}$ denotes $\|f\|_{X(\mr^2)}$.
	\end{enumerate}
Here, $X$ may denote any of these function spaces: homogeneous Besov space $\dot{B}^s_{p,r}$, nonhomogeneous Besov space $B^s_{p,r}$, Lebesgue space $L^p$, or Sobolev space $W^{s,p}$.
\end{itemize}

\section{Preliminaries}\label{sec:2}

\subsection{Key analytical estimates on Besov spaces}\label{sec:2.1}
For the definition and basic properties of Besov spaces, we refer to Appendix \ref{sec:appendixA}. 

In what follows, we shall establish a logarithmic-type inequality characterizing the difference between the homogeneous Besov spaces $\dot{B}^{s}_{p,r_1}$ and $\dot{B}^{s}_{p,r_2}$. This result will be particularly useful for estimating the function $f_N$ in Besov spaces in Proposition \ref{fN-sequence}.
\begin{lemma}\label{besov-log-lemma}
	There exists a constant $C>0$ such that for all $s\in\mathbb{R}, \nu>0, n\geq 2,  1\leq r_1\leq \min\{r_2,r_3\}$ and $1\leq p\leq \infty$, 
	\begin{align}\label{besov-embed-log}
		\|f\|_{\dot{B}^s_{p,r_1}(\mr^n)}\leq C\dfrac{1+\nu}{\nu}\left(1+\|f\|_{\dot{B}^s_{p,r_2}(\mr^n)}\left[\log_2\left(\e+\|f\|_{\dot{B}^{s-\nu}_{p,r_3}(\mr^n)}+\|f\|_{\dot{B}^{s+\nu}_{p,r_3}(\mr^n)}\right)\right]^{\frac{1}{r_1}-\frac{1}{r_2}}\right).
	\end{align}
\end{lemma}
\begin{proof}
	Using the definition of homogeneous Besov norms, we decompose the norm into three parts:
	\[
	\|f\|_{\dot{B}^s_{p,r_1}}^{r_1} = \underbrace{\sum_{k\leq -\sigma} \left(2^{ks}\|\dot{\Delta}_{k}f\|_{L^p}\right)^{r_1}}_{I_1} 
	+ \underbrace{\sum_{|k|<\sigma} \left(2^{ks}\|\dot{\Delta}_{k}f\|_{L^p}\right)^{r_1}}_{I_2} 
	+ \underbrace{\sum_{k\geq \sigma} \left(2^{ks}\|\dot{\Delta}_{k}f\|_{L^p}\right)^{r_1}}_{I_3},
	\]
	where $\sigma > 0$ will be chosen later.
	
	\textbf{Estimate of $I_2$:} By H\"older's inequality with exponents $\frac{r_2}{r_1}$ and $\frac{r_2}{r_2-r_1}$, we obtain:
	\[
	I_2 \leq (2\sigma-1)^{\frac{r_2-r_1}{r_2}} \|f\|_{\dot{B}^s_{p,r_2}}^{r_1}.
	\]
	
	\textbf{Estimate of $I_1$ and $I_3$:} For any $\nu > 0$, we have
	\[
	I_1 \leq \left(\sum_{k\leq -\sigma} 2^{k\nu \frac{r_1r_3}{r_3-r_1}}\right)^{\frac{r_3-r_1}{r_3}} \|f\|_{\dot{B}^{s-\nu}_{p,r_3}}^{r_1} 
	\leq \frac{2^{-\sigma\nu r_1}}{\left(1-2^{-\nu \frac{r_1r_3}{r_3-r_1}}\right)^{\frac{r_3-r_1}{r_3}}} \|f\|_{\dot{B}^{s-\nu}_{p,r_3}}^{r_1}.
	\]
	Similarly,
	\[
	I_3 \leq \frac{2^{-\sigma\nu r_1}}{\left(1-2^{-\nu \frac{r_1r_3}{r_3-r_1}}\right)^{\frac{r_3-r_1}{r_3}}} \|f\|_{\dot{B}^{s+\nu}_{p,r_3}}^{r_1}.
	\]
	By	combining these estimates, we get 
	\[
	\|f\|_{\dot{B}^s_{p,r_1}} \leq (2\sigma-1)^{\frac{r_2-r_1}{r_1r_2}} \|f\|_{\dot{B}^s_{p,r_2}} 
	+ \frac{2^{-\sigma\nu}}{\left(1-2^{-\nu \frac{r_1r_3}{r_3-r_1}}\right)^{\frac{r_3-r_1}{r_1r_3}}} \left(\|f\|_{\dot{B}^{s-\nu}_{p,r_3}} + \|f\|_{\dot{B}^{s+\nu}_{p,r_3}}\right).
	\]
	Choosing $\sigma$ being a closest positive integer to $\dfrac{1}{\nu}\log_2\left(\e+\|f\|_{\dot{B}^{s-\nu}_{p,r_3}}+\|f\|_{\dot{B}^{s+\nu}_{p,r_3}}\right)$ yields \eqref{besov-embed-log}.
\end{proof}

At the conclusion of this subsection, we recall an important commutator estimate and establish two key product estimates. The first estimate concerns a function interacting with a classical Riesz operator, while the second deals with a function that requires no integration.
\begin{lemma}[\cite{GTM-fourier-grafakos}]\label{riesz-commutator}
	Let $n\geq 2$ and $\mathcal{R}$ be a Riesz operator, the following estimate are valid,
	\begin{align*}
		\left\|\mathcal{R}(f g)-\mathcal{R}(f) g\right\|_{L^p(\mr^n)}&\leq C\|f\|_{BMO(\mr^n)}\|g\|_{L^p(\mr^n)}\\
		&\leq C\|f\|_{L^\infty(\mr^n)}\|g\|_{L^p(\mr^n)}.
	\end{align*}
	where $1<p<\infty$ and $C$ depends only on $p,n$.
\end{lemma}

\begin{lemma}\label{riesz-product}
	Let $n\geq 2, 1\leq p<\infty$ and $s>0$.  Let $\mathcal{R}$ be a Riesz operator, the following estimates hold, for any functions $f,g\in L^\infty(\mr^n)\cap  B^s_{p,1}(\mr^n)$
	\begin{align}\label{product-1-riesz}
			\left\|f \mathcal{R} g\right\|_{B^s_{p,1}(\mr^n)}\leq C\left(\|f\|_{L^\infty(\mr^n)}\|g\|_{ B^s_{p,1}(\mr^n)}+\|g\|_{L^\infty(\mr^n)}\|f\|_{B^s_{p,1}(\mr^n)}\right),
	\end{align}
and \begin{align}\label{product-2-infty}
		\left\|f g\right\|_{B^s_{p,1}(\mr^n)}\leq C\left(\|f\|_{L^\infty(\mr^n)}\|g\|_{ B^s_{p,1}(\mr^n)}+\|g\|_{L^p(\mr^n)}\|f\|_{B^s_{\infty,1}(\mr^n)}\right).
\end{align}

\end{lemma}
\begin{proof}
By  Bony's decomposition, we have
	\begin{align}\label{bony-besov}
		f\mathcal{R}g=T_f \mathcal{R}g+ T_{\mathcal{R}g}f + R(f,\mathcal{R}g).
	\end{align}
By Lemma \ref{B}, we obtain	
	\begin{equation}\label{low}\begin{split}
			\left\|T_f \mathcal{R} g\right\|_{ B^s_{p,1}}&\leq C\|f\|_{L^\infty}\|g\|_{B^{s}_{p,1}},\end{split}
	\end{equation}
and for $s>0,$
\begin{equation}\label{Reminder}
	\begin{split}
		\left\|R(f,\mathcal{R}g)\right\|_{B^s_{p,1}}&\leq C\|f\|_{B^0_{\infty,\infty}}\|\mathcal{R}g\|_{ B^s_{p,1}}\\&\leq C\|f\|_{L^\infty}\|g\|_{ B^s_{p,1}}
		.\end{split}
\end{equation}	

It remains to estimate 	 $T_{\mathcal{R}g}f$. Indeed,
	\begin{equation}\label{f-high-frequency}
		\begin{split}
			\left\|T_{\mathcal{R} g}f \right\|_{ B^s_{p,1}}&=\sum_{j\geq -1}2^{js}\left\| \Delta_{j}\left(\sum_{|k-j|\leq 5}S_{k-1}\mathcal{R}g\Delta_{k}f\right)\right\|_{L^p}\\
			&\leq \left\| \Delta_{-1}\left(\sum_{|k+1|\leq 5}S_{k-1}\mathcal{R}g\Delta_{k}f\right)\right\|_{L^p}+ \sum_{j\geq 0}2^{js}\left\| \Delta_{j}\left(\sum_{|k-j|\leq 5}S_{k-1}\mathcal{R}g\Delta_{k}f\right)\right\|_{L^p}\\ &\leq \left\| \Delta_{-1}\left(\sum_{|k+1|\leq 5}S_{k-1}\mathcal{R}g\Delta_{k}f\right)\right\|_{L^p}+\sum_{j\geq 0}2^{js}\left\| \Delta_{j}\mathcal{R}\left(\sum_{|k-j|\leq 5}S_{k-1}g\Delta_{k}f\right)\right\|_{L^p}\\
			&~~~~~+\sum_{j\geq 0}2^{js}\left\| \Delta_{j}\mathcal{R}\left(\sum_{|k-j|\leq 5}S_{k-1}g\Delta_{k}f\right)-\Delta_{j}\left(\sum_{|k-j|\leq 5}S_{k-1}\mathcal{R}g\Delta_{k}f\right)\right\|_{L^p}.\end{split}
	\end{equation}
	Applying Lemma \ref{riesz-commutator} yields
	\begin{align*}
		&\sum_{j\geq 0}2^{js}\left\| \Delta_{j}\mathcal{R}\left(\sum_{|k-j|\leq 5}S_{k-1}g\Delta_{k}f\right)-\Delta_{j}\left(\sum_{|k-j|\leq 5}S_{k-1}\mathcal{R}g\Delta_{k}f\right)\right\|_{L^p}\\
		&\leq  \sum_{j\geq 0}2^{js}\left\| \mathcal{R}\left(\sum_{|k-j|\leq 5}S_{k-1}g\Delta_{k}f\right)-\left(\sum_{|k-j|\leq 5}S_{k-1}\mathcal{R}g\Delta_{k}f\right)\right\|_{L^p}\\
		&\leq  C\sum_{j\geq 0}2^{js}\sum_{|k-j|\leq 5} \|S_{k-1}g\|_{L^\infty}\|\Delta_{k}f\|_{L^p}\\
		&\leq C\sum_{j\geq -1}2^{js} \|g\|_{L^\infty}\|\Delta_{j}f\|_{L^p}\leq C\|g\|_{L^\infty}\|f\|_{ B^s_{p,1}}.
	\end{align*}
	In addition, 
	\begin{equation*}\begin{split}
		\sum_{j\geq 0}2^{js}\left\| \Delta_{j}\mathcal{R}\left(\sum_{|k-j|\leq 5}S_{k-1}g\Delta_{k}f\right)\right\|_{L^p}
		&\leq C \sum_{j\geq -1}2^{js}\left\| S_{j-1}g\Delta_{j}f\right\|_{L^p}\\
		&\leq C \|g\|_{L^\infty}\sum_{j\geq -1}2^{js}\|\Delta_{j}f\|_{L^p}\\&\leq C\|g\|_{L^\infty}\|f\|_{ B^s_{p,1}}, 
\end{split}	\end{equation*} and
	\begin{equation*}
		\begin{split}
			\left\| \Delta_{-1}\left(\sum_{|k+1|\leq 5}S_{k-1}\mathcal{R}g\Delta_{k}f\right)\right\|_{L^p}
			\leq C\|g\|_{L^p}\|f\|_{L^\infty}.
		\end{split}
	\end{equation*}
	Combining these estimates with \eqref{f-high-frequency} gives to
	\begin{align}\label{high}
		\left\|T_{\mathcal{R} g}f \right\|_{ B^s_{p,1}}\leq C\left(\|g\|_{ B^s_{p,1}}\|f\|_{L^\infty}+\|g\|_{L^\infty}\|f\|_{ B^s_{p,1}}\right).
	\end{align}
	Collecting 	 \eqref{low},  \eqref{Reminder} and \eqref{high} with \eqref{bony-besov} yields  \eqref{product-1-riesz}.
	
To establish   estimate \eqref{product-2-infty}, we  adopt to  Bony's paraproduct decomposition again
	\begin{equation*}
		fg=T_f g + T_g f + R(f,g).
	\end{equation*}
The bounds for the paraproduct term  $T_f g$ and the remainder term $R(f,g)$ follow directly from Lemma \ref{B}, which yields
	\begin{equation*}
		\|T_f g\|_{B^s_{p,1}} \leq C\|f\|_{L^\infty}\|g\|_{B^s_{p,1}} \quad \text{and} \quad \|R(f,g)\|_{B^s_{p,1}} \leq C\|f\|_{L^\infty}\|g\|_{B^s_{p,1}}.
	\end{equation*}
The essential part of the proof lies in estimating the remaining paraproduct term $T_g f$. This term can be estimated as follows:
	\begin{equation*}
		\begin{split}
			\left\|T_{  g}f \right\|_{ B^s_{p,1}}&=\sum_{j\geq -1}2^{js}\left\| \Delta_{j}\left(\sum_{|k-j|\leq 5}S_{k-1} g\Delta_{k}f\right)\right\|_{L^p}\\
			&\leq C \sum_{j\geq -1}2^{js}\left\| S_{j-1}g\Delta_{j}f\right\|_{L^p}\\
			&\leq C \|g\|_{L^p}\sum_{j\geq -1}2^{js}\|\Delta_{j}f\|_{L^\infty}\leq C\|g\|_{L^p}\|f\|_{B^s_{\infty,1}}. 
		\end{split}
	\end{equation*}
Therefore, the proof of \eqref{product-2-infty} is finished by collecting the above estimates.
\end{proof}

\subsection{Technical lemmas of composite functions}\label{sec:2.2}

\begin{lemma}[\cite{bahouriFourierAnalysisNonlinear2011, miao2019-littlewood}]\label{lagrange-flowmap}Let $\Phi:\mr^n\times[0,\infty)\rightarrow\mr^n$
	be the flow map generated by a velocity field 
$u$, defined as the solution to
	\begin{align}\label{2.4}
		\begin{cases}
			\partial_t\Phi({\bf x},t)=u(\Phi({\bf x},t),t),\\
			\Phi({\bf x},0)={\bf x},
		\end{cases}
	\end{align}which admits the integral formulation
	\begin{align}\label{2.4-integral}
		\Phi({\bf x},t)={\bf x}+\int_{0}^t u(\Phi({\bf x},s),s)\dd s.
	\end{align}
Then, the flow map and its inverse $\Phi^{-}$	satisfy the following estimates for any $t>0,$
	\begin{align*}
		\|\Phi^{\pm}\|_{L^\infty\left(0,t; Lip(\mr^n)\right)}\leq \|\nabla \Phi^{\pm}\|_{L^\infty\left(0,t; L^\infty(\mr^n)\right)}\leq \exp\left(\int_{0}^t\|\nabla u(s)\|_{L^\infty(\mr^n)}\dd s\right),\\
		\|\Phi^{\pm}-I d\|_{L^\infty\left(0,t; Lip(\mr^n)\right)}\leq 	\|\Phi^{\pm}\|_{L^\infty\left(0,t; Lip(\mr^n)\right)} \int_{0}^t\|u(s)\|_{Lip(\mr^n)}\dd s,
	\end{align*}
	where $Id$ is the identity matrix and $Lip$ denotes the space of Lipschitz functions with norm
	\begin{align*}
		\|\Phi(\cdot)\|_{Lip(\mr^n)}\define \sup_{{\bf x}\neq {\bf y}}\dfrac{|\Phi({\bf x})-\Phi({\bf y})|}{|{\bf x}-{\bf y}|}.
	\end{align*}
\end{lemma}

\begin{lemma}[\cite{bahouriFourierAnalysisNonlinear2011, miao2019-littlewood}]\label{phi-composite}
	Let $n\geq 2, 1\leq p\leq \infty$ and $0<\alpha<1.$ Suppose $\Phi$ is a bi-Lipschitz measure-preserving  mapping from $\mr^n$ to $\mr^n$.  Then there exist  constants $C=C(n)>0$
 such that for any function $u$, the following estimates hold
	\begin{align*}
		\|u\circ \Phi\|_{\dot{B}^\alpha_{p,1}(\mr^n)}\leq C \|\nabla \Phi\|_{L^\infty(\mr^n)}^\alpha \|u\|_{\dot{B}^\alpha_{p,1}(\mr^n)}
	\end{align*}
	and
	\begin{align*}
		\|u\circ \Phi\|_{B^0_{\infty,1}(\mr^n)}\leq C\left(1+\log\left(\|\Phi\|_{Lip(\mr^n)}\|\Phi^{-1}\|_{Lip(\mr^n)}\right)\right)\|u\|_{B^0_{\infty,1}(\mr^n)}.
	\end{align*}

\end{lemma}

\begin{lemma}[\cite{elgindi2020infty}]\label{commutator-elgindi}
	Let $n\geq 2, 0<\alpha<1, 1<p<\infty$ and $\Phi$ be a bi-Lipschitz  measure-preserving mapping from $\mr^n$ to $\mr^n$. Define the  commutator $\left[\mathcal{R},\Phi\right]$ as 
	\begin{align*}
		\left[\mathcal{R},\Phi\right]f=\left(\mathcal{R}(f)\circ\Phi\right)-\mathcal{R}\left(f\circ\Phi\right),
	\end{align*}
	where $\mathcal{R}$ is  a Calder\'{o}n-Zygmund singular integral operator with  a kernel $K(x)=\frac{\Omega(\frac{x}{|x|})}{|x|^n}$
for some $\Omega$ which is mean-zero on the unit-sphere.
Then $$\left[\mathcal{R},\Phi\right] : B^\alpha_{p,1}\rightarrow B^\alpha_{p,1}$$ is bounded. Furthermore, there exists a universal constant $C>0$ depending only upon $\mathcal{R}$, $\left\| \Phi\right\|_{Lip}$, $\left\| \Phi^{-1}\right\|_{Lip}$ and  $n$ such that
	\begin{align*}
		\|[\mathcal{R}, \Phi] f\|_{B_{p, 1}^\alpha(\mr^n)} \leq C \max \left\{\|\Phi-I d\|_{Lip(\mr^n)},\left\|\Phi^{-1}-I d\right\|_{Lip(\mr^n)}\right\}\|f\|_{B_{p, 1}^\alpha(\mr^n)}.
	\end{align*}

\end{lemma}

\section{Construction of special functions}\label{sec:3}

Building upon the seminal work of Elgindi and Masmoudi \cite{elgindi2020infty}, who established a framework for characterizing the discontinuity of Riesz operators in 
$L^\infty$
via logarithmic singularities, our work extends this analysis to general steady-state solutions where $g'(x_2)$
is not identically zero. Unlike the Riesz operator approach in \cite{elgindi2020infty}, our setting necessitates a fundamentally different analytical framework due to the non-trivial structure of the steady states. A key technical challenge lies in deriving sharp lower bound estimates that incorporate a carefully constructed non-vanishing weight function. This weight function must simultaneously capture the underlying steady-state geometry and preserve the delicate balance required for logarithmic singularity estimates. 
 
 
\begin{prop}\label{fN-sequence}
	Let $2< p<\infty$.
	There exist a sequence of $C^\infty_c(\mr^2)$ functions $\{f_N\}_{N=1}^{\infty}$ and constants $c_s'>0$ (independent of $N$) such that the following estimates hold
	\begin{align}\label{2.1}
	\| f_N(\cdot)\|_{W^{1,\infty}(\mr^2)}&\leq c,\\\label{2.0}
	\|f_N(\cdot)\|_{B^{\frac{2}{p}}_{p,1}(\mr^2)}&\leq c\left(\log N\right)^{\frac{1}{2}},\\\label{2.2}
	\|f_N(\cdot)\|_{B^{\frac{2}{p}+1}_{p,1}(\mr^2)}&\leq cN.
\end{align}

Furthermore,  if 
$h({\bf x})$ is a continuous function that is not identically zero,
then there exists a constant $c_1>0,$ independent of $N,$ such that
\begin{align}\label{2.3.1}
	\|h(\cdot) \mathcal{R}^2_{1}\nabla^\perp f_N(\cdot)\|_{L^\infty(\mr^2)}&\geq c_1N,
\end{align}
where $\mathcal{R}_{1}=\partial_{x_1} \Lambda^{-\frac{1}{2}}$ is the first Riesz transform.
\end{prop}

\begin{proof}
Let $P(x_1,x_2)=x_1^3x_2-x_1x_2^3$ be a  harmonic polynomial of degree 4 (i.e., $\Delta P=0$).
Define
\begin{align*}
	G_N(x_1,x_2)=P(x_1,x_2)\log(x_1^2+x_2^2+\frac{1}{2^N}).
\end{align*}
A direct calculation shows that
\begin{align*}
	\Delta \log(x_1^2+x_2^2+\frac{1}{2^N})&=\dfrac{\frac{4}{2^N}}{\left(x_1^2+x_2^2+\frac{1}{2^N}\right)^2},
\end{align*}and consequently,
\begin{align}\nonumber
	&\Delta G_N(x_1,x_2)\\\label{delta-G}
	&=\dfrac{16x_1^3x_2-16x_1x_2^3}{x_1^2+x_2^2+\frac{1}{2^N}}+\dfrac{1}{2^{N-2}}\dfrac{x_1^3x_2-x_1x_2^3}{\left(x_1^2+x_2^2+\frac{1}{2^N}\right)^2}.
\end{align}
{\bf Derivative Analysis and Scaling Behavior:}
Taking the first-order partial derivative with respect to $x_1$ on
\eqref{delta-G} yields
\begin{align*}
	\partial_{1}\Delta G_N(x_1,x_2)&=8\dfrac{3x_1^2x_2-x_2^3}{x_1^2+x_2^2+\frac{1}{2^N}}-16\dfrac{x_1^4x_2-x_1^2x_2^3}{\left(x_1^2+x_2^2+\frac{1}{2^N}\right)^2}\\
	&~~~+4\dfrac{\dfrac{1}{2^{N}}(3x_1^2x_2-x_2^3)}{\left(x_1^2+x_2^2+\frac{1}{2^N}\right)^2}-16\dfrac{\dfrac{1}{2^{N}}\left(x_1^4x_2-x_1^2x_2^3\right)}{\left(x_1^2+x_2^2+\frac{1}{2^N}\right)^3}.
\end{align*}By performing scaling analysis  for the terms appearing on the right-hand side, we derive 
\begin{equation}\label{fisrt}
|\partial_{1}\Delta G_N(x_1,x_2)|\leq C \left(x_1^2+x_2^2+\frac{1}{2^N}\right)^\frac{1}{2}.
\end{equation}
To verify this, we apply Young's inequality to bound the numerator terms as follows:
$$3x_1^2x_2-x_2^3\leq C(x_1^2+x_2^2+\dfrac{1}{2^{N}})^\frac{3}{2},$$
$$x_1^4x_2-x_1^2x_2\leq  C(x_1^2+x_2^2+\dfrac{1}{2^{N}})^\frac{5}{2},$$
\begin{equation*}
	\begin{split}
		\dfrac{1}{2^{N}}(3x_1^2x_2-x_2^3)&\leq C\left((\dfrac{1}{2^{N}})^\frac{5}{2}+ x_1^{\frac{10}{3}}x_2^{\frac{5}{3}}+x_2^5\right)\\&\leq C(x_1^2+x_2^2+\dfrac{1}{2^{N}})^\frac{5}{2},
\end{split}\end{equation*}
\begin{equation*}
	\begin{split}\dfrac{1}{2^{N}}(x_1^4x_2-x_1^2x_2^3)&\leq C\left((\dfrac{1}{2^{N}})^\frac{7}{2}+ x_1^{\frac{28}{5}}x_2^{\frac{7}{5}}+x_1^{\frac{14}{5}}x_2^{\frac{21}{5}}\right)\\&\leq C(x_1^2+x_2^2+\dfrac{1}{2^{N}})^\frac{7}{2}.
	\end{split}
\end{equation*}
In a similarly way to get \eqref{fisrt},  we obtain
\begin{equation}\label{fisrt-2}|\partial_{2}\Delta G_N(x_1,x_2)|\leq C \left(x_1^2+x_2^2+\frac{1}{2^N}\right)^\frac{1}{2}.
\end{equation} 
By explicitly expressing the second-order, third-order and fourth-order derivatives of $\Delta G_N$ and performing a similar scaling analysis, 
for $\forall$ $i,j,k,l\in\{1, 2\}$,  we derive
\begin{equation}\label{QN-234-degree}
	\begin{aligned}
		&|\partial_i\partial_j\Delta G_N(x_1,x_2)|\leq C,\\
		&|\partial_i\partial_j\partial_k\Delta G_N(x_1,x_2)|\leq \dfrac{C}{\left(x_1^2+x_2^2+\frac{1}{2^N}\right)^\frac{1}{2}},\\
		&|\partial_i\partial_j\partial_k\partial_l\Delta G_N(x_1,x_2)|\leq \dfrac{C}{x_1^2+x_2^2+\frac{1}{2^N}}.
	\end{aligned}
\end{equation}
{\bf Norm Estimates:}
These scaling properties above imply the following $L^2$
bounds on $B({\bf 0};2)$
\begin{equation}\label{0-orderL2}
	\begin{split}
	\left\|\partial_i\Delta G_N\right\|_{L^2(B({\bf 0};2))} \leq& \left\|\left(x_1^2+x_2^2+\frac{1}{2^N}\right)^\frac{1}{2}\right\|_{L^2(B({\bf 0};2))}\leq C .
	\end{split}
\end{equation}
\begin{equation}\label{1-orderL2}
	\begin{split}
	\left\|\partial_i\partial_j\Delta G_N\right\|_{L^2(B({\bf 0};2))} \leq C.
	\end{split}
\end{equation}
\begin{equation}\label{2-orderL2}
	\begin{split}
\left\|\partial_i\partial_j\partial_k\Delta G_N\right\|_{L^2(B({\bf 0};2))}\leq&	\left\|\dfrac{1}{\left(x_1^2+x_2^2+\frac{1}{2^N}\right)^\frac{1}{2}}\right\|_{L^2(B({\bf 0};2))}\leq C\sqrt{N}.
	\end{split}
\end{equation}
\begin{equation}\label{QN-234-degree-L2norm}
	\begin{split}
		\left\|\partial_i\partial_j\partial_k\partial_l\Delta G_N\right\|_{L^2(B({\bf 0};2))}&\leq
	\left\|\dfrac{1}{x_1^2+x_2^2+\frac{1}{2^N}}\right\|_{L^2(B({\bf 0};2))}	
		\leq C2^\frac{N}{2}.
	\end{split}
\end{equation}
Moreover, we have the  uniform $L^\infty$
bounds for first and second derivatives:
\begin{align}\label{GN-L-infty}
	\left\|	\partial_i \Delta G_N\right\|_{L^\infty (B({\bf 0};2))}+\left\|	\partial_i \partial_j \Delta G_N\right\|_{L^\infty (B({\bf 0};2)) }\leq C.
\end{align}
{\bf Singular Behavior:}
However, a special fourth-order derivative of $G_N$ exhibits logarithmic singularity
\begin{align*}
	\partial_{1}^3	\partial_{2}G_N(x_1,x_2)=6\log(x_1^2+x_2^2+\dfrac{1}{2^N})+G^1_N(x_1,x_2),
\end{align*}
where the remainder term $G^1_N$ satisfies the uniform bound $$\left\|G^1_N\right\|_{L^\infty (B({\bf 0};2)) }\leq C.$$ This singular behavior yields the crucial estimate for sufficiently large  $N$
\begin{equation}\label{GNx1-3,x2}
	\begin{split}
	\left\|\partial_{1}^3	\partial_{2}G_N \right\|_{L^\infty(B({\bf 0};1))}&\geq  6|\log(\dfrac{1}{2^N})|-\left\|G_N^1(\cdot)\right\|_{L^\infty B({\bf 0};2)}\\ &\geq  CN.\end{split}
\end{equation}
{\bf Construction of $f_N$:}
Let $\chi\in C_{c}^{\infty}(\mr^2)$ be a smooth cut-off function satisfying
\begin{align*}
	\chi=1 \text{~on~} B({\bf 0}; 1),\quad  \chi=0 \text{~on~} B({\bf 0}; 2)^c.
\end{align*}
Since $h({\bf x})\not\equiv  0$,  there exists  ${\bf x_0}$ such that $h(\bf x_0) \neq 0$.
We define \begin{align*}
	f_N({\bf x})=\partial_{x_1}\Delta \left(G_N({\bf x}-{\bf x_0})\chi ({\bf x}-{\bf x_0})\right),
\end{align*}
where ${\bf x}=(x_1,x_2)$ and $G_N$ is the function constructed above. It is obvious  that each $f_N$ belongs to  $C_{c}^{\infty}(\mr^2)$.\\[1mm]
{\bf Decomposition and Estimates:}
The function $f_N$ can be decomposed into two distinct parts:
\begin{align}\label{fN-two-parts}
f_N({\bf x})= \partial_{x_1}\Delta \big(G_N({\bf x}-{\bf x_0})\big)\chi ({\bf x}-{\bf x_0})+\widetilde{f}_N({\bf x}),
\end{align}
where   $\widetilde{f}_N({\bf x})$ is the remainder term, supported in the annulus, i.e.,
\begin{align*}
	 \text{supp} \widetilde{f}_N({\bf x})\subset \{{\bf x}; ~1\leq |{\bf x}-{\bf x_0}|\leq 2\}.
\end{align*}
As a result, $G_N$ and any order derivative of $G_N$ have a uniform upper bound in $\text{supp} \widetilde{f}_N({\bf x})$, which implies that, for $m=0,1,2,3$
\begin{align}\label{fNrest-uniform-upperbound}
		\| \nabla^m\widetilde{f}_N(\cdot)\|_{L^2(\mr^2)}, \|\nabla^m \widetilde{f}_N(\cdot)\|_{L^\infty(\mr^2)}\leq C.
\end{align}
{\bf $W^{1,\infty}$-estimate:}
Combining \eqref{fN-two-parts} with \eqref{GN-L-infty} and \eqref{fNrest-uniform-upperbound}, we obtain $$\|f_N(\cdot)\|_{W^{1,\infty}(\mr^2)}\leq  c,$$ which finishes the proof of \eqref{2.1}.\\[1mm]
{\bf $H^{3}$-estimate:}
Using \eqref{0-orderL2}-\eqref{QN-234-degree-L2norm}, \eqref{fN-two-parts} and  \eqref{fNrest-uniform-upperbound} yields
\begin{equation}\label{0-fN}
	\begin{split}
		\|f_N(\cdot)\|_{L^2(\mr^2)}&\leq C.
	\end{split}
\end{equation}
\begin{equation}\label{1-fN}
	\begin{split}
		\|\nabla f_N(\cdot)\|_{L^2(\mr^2)}&\leq C.
	\end{split}
\end{equation}
\begin{equation}\label{2-fN}
	\begin{split}
	\|\nabla^2 f_N(\cdot)\|_{L^2(\mr^2)}&\leq C\sqrt{N}.
	\end{split}
\end{equation}
\begin{equation}\label{3-fN}
	\begin{split}
\|\nabla^3 f_N(\cdot)\|_{L^2(\mr^2)}&\leq C2^\frac{N}{2}.
\end{split}\end{equation}
{\bf Besov Space Bounds:}
By virtue of the continuous    embedding  $B^2_{2,1}\hookrightarrow B^{\frac{2}{p}+1}_{p,1}$ combined with Lemma \ref{lem:besov_compact}, we obtain
\begin{equation*}\begin{split}
	&\|f_N(\cdot)\|_{B^{\frac{2}{p}+1}_{p,1}}\\ \leq
	&	C\|f_N(\cdot)\|_{\dot{B}^2_{2,1}}
	\\ \leq&
	C\|f_N(\cdot)\|_{\dot{H}^2}\left(1+\log^\frac{1}{2}\left(\dfrac{\|f_N(\cdot)\|_{\dot{H}^1}+\|f_N(\cdot)\|_{\dot{H}^3}}{\|f_N(\cdot)\|_{\dot{H}^2}}\right)\right),\end{split}
\end{equation*}
where the last inequality follows from an application of Lemma \ref{besov-log-lemma}.
Therefore, plugging \eqref{1-fN}-\eqref{3-fN} into the above inequality gives to
\begin{equation*}\begin{split}
		\|f_N(\cdot)\|_{B^{\frac{2}{p}+1}_{p,1}} \leq
		&C\sqrt{N}\left(1+\left(\log \dfrac{C+2^\frac{N}{2}}{\sqrt{N}}\right)^\frac{1}{2}\right)\\ \leq& C\sqrt{N}\left(1+C\sqrt{N}\right)\leq cN,
	\end{split}
\end{equation*}which finishes the proof of \eqref{2.2}.
Similarly,
\begin{align*}
	&\|f_N(\cdot)\|_{B^{\frac{2}{p}}_{p,1}}\\ \leq&
	C	\|f_N(\cdot)\|_{\dot{B}^1_{2,1}}\\ \leq&
	C\|f_N(\cdot)\|_{\dot{B}^1_{2,2}}\left(1+\log^\frac{1}{2}\left(\dfrac{\|f_N(\cdot)\|_{\dot{B}^0_{2,2}}+\|f_N(\cdot)\|_{\dot{B}^2_{2,2}}}{\|f_N(\cdot)\|_{\dot{B}^1_{2,2}}}\right)\right)\\ 
	 \leq& C\left(1+\left(\log N\right)^{\frac{1}{2}}\right)\leq c\left(\log N\right)^{\frac{1}{2}},
\end{align*}
which  proves \eqref{2.0}.\\[1mm]
{\bf Key Lower Bound:}
Since  $\mathcal{R}_{1}^2\partial_{1}\Delta =-\partial_{1}^3$, we have
\begin{align*}
	\mathcal{R}_{1}^2\nabla^{\perp} f_N({\bf x})=-\nabla^\perp \partial_{1}^3\big(G_N({\bf x}-{\bf x_0})\chi ({\bf x}-{\bf x_0})\big).
\end{align*}
Then, for sufficiently large $N,$ we have
\begin{equation*}\begin{split}
		&\left\|h(\cdot)\mathcal{R}_{1}^2\nabla^\perp f_N(\cdot)\right\|_{L^\infty}\\ \geq&\left\|h(\cdot)\partial_{1}^3\partial_{2}\big(G_N({\bf x}-{\bf x_0})\chi ({\bf x}-{\bf x_0})\big)\right\|_{L^\infty(B({\bf x_0},1))}\\
		\geq & \left\|h(\cdot)\partial_{1}^3\partial_{2}\big(G_N({\bf x}-{\bf x_0})\big)\chi ({\bf x}-{\bf x_0})\right\|_{L^\infty(B({\bf x_0},1))}-C\\
		= & \left\|h(\cdot)\partial_{1}^3\partial_{2}G_N({\bf x}-{\bf x_0})\right\|_{L^\infty(B({\bf x_0},1))}-C\\
		\geq &CN|h({\bf x_0})|-C\geq c_1N,
	\end{split}
\end{equation*} where in the last inequality  \eqref{GNx1-3,x2} has been used.
Hence, \eqref{2.3.1} is obtained  and  the proof of this lemma is also completed.
\end{proof}

\section{Local existence and uniqueness}\label{sec:4}
In this section, we are going to establish the local existence and uniqueness  of system \eqref{p-IPM}. We first state a Gronwall-type inequality.
\begin{lemma}\label{gronwall}
	Let $a(t)$ and $f(t)$ be non-negative continuous functions on $[t_0, t_1]$. Let  $\alpha>0$ and $C$ be a non-negative constant. If $f(t)$ satisfies, for any $t \in [t_0, t_1]$, 
	\begin{align}\label{gron1}
		f(t)\leq C+\int_{t_0}^{t}a(s)f^\alpha(s)\dd s,
	\end{align}
	then, when $\alpha = 1$, it holds 
	\begin{align}\label{gron-alpha=1}
		f(t)\leq C\exp \left(\int_{t_0}^{t}a(s)\dd s\right).
	\end{align}
	On the other hand, when $\alpha\neq 1$, there holds
	\begin{align}\label{gron-alpha neq 1}
		f(t)\leq \left[C^{1-\alpha}+(1-\alpha)\int_{t_0}^ta(s)\dd s\right]^{\frac{1}{1-\alpha}}.
	\end{align}
\end{lemma}
\begin{proof}
	For $\alpha=1$,
	\begin{align*}
		\dfrac{\text{d}}{\text{d} t}\left(C+\int_{t_0}^{t}a(s)f(s)\dd s\right)\leq a(t)\left(C+\int_{t_0}^{t}a(s)f(s)\dd s\right).
	\end{align*}
	By applying the chain rule for derivative,
	\begin{align*}
		\dfrac{\text{d}}{\text{d} t}\ln \left(C+\int_{t_0}^{t}a(s)f(s)\dd s\right)\leq a(t).
	\end{align*}
	Integrate both sides from $t_0$ to $t$
	\begin{align*}
		\ln \left(C+\int_{t_0}^{t}a(s)f(s)\dd s\right)-\ln C\leq \int_{t_0}^{t}a(s)\dd s.
	\end{align*}
	Thus,
	\begin{align*}
		C+\int_{t_0}^{t}a(s)f(s)\dd s\leq C\exp\left(\int_{t_0}^{t}a(s)\dd s\right)
	\end{align*}
	which further implies  \eqref{gron-alpha=1}, after combining \eqref{gron1}. 
	
	For $\alpha\neq 1$, let $A(t)=\int_{t_0}^{t}a(s)f^\alpha(s)\dd s$, then
	\begin{align*}
		\dfrac{A'(t)}{\left(C+A(t)\right)^{\alpha}}=\dfrac{a(t)f^\alpha(t)}{\left(C+A(t)\right)^{\alpha}}\leq a(t).
	\end{align*}
	Integrate both sides from $t_0$ to $t$
	\begin{align*}
		\dfrac{\left(C+A(t)\right)^{1-\alpha}}{1-\alpha}-
		\dfrac{C^{1-\alpha}}{1-\alpha}\leq \int_{t_0}^{t}a(s)\dd s.
	\end{align*}
	Therefore, by reusing \eqref{gron1},
	\begin{align*}
		\dfrac{f(t)^{1-\alpha}}{1-\alpha}\begin{cases}
			\leq \dfrac{C^{1-\alpha}}{1-\alpha}+\int_{t_0}^{t}a(s)\dd s,&0<\alpha<1,\\
			\geq \dfrac{C^{1-\alpha}}{1-\alpha}+\int_{t_0}^{t}a(s)\dd s,&\alpha>1.
		\end{cases}
	\end{align*}
	Finally, \eqref{gron-alpha neq 1} holds for both $\alpha > 1$ and $0<\alpha < 1$.
\end{proof}

We are ready to prove the local
existence and uniqueness in a suitable functional setting.

\begin{prop}\label{local}
	Let $\eta_0\in B^{\frac{2}{p}+1}_{p,1}$ with $2< p<\infty$.
	Then, there exists   a unique solution $\eta$ of system \eqref{p-IPM}  satisfying $\eta \in L^\infty(0,T; B^{\frac{2}{p}+1}_{p,1})$ for some $T>0$. Moreover, for $t\in[0,T],$ the following estimates hold
	\begin{align*}
		\lef \eta(t)\rig_{B^{\frac{2}{p}+1}_{p,1}}\leq \dfrac{\frac{c_g}{2}+2\left\| \eta_0\right\|_{B^{\frac{2}{p}+1}_{p,1}}}{1-C\left(\frac{c_g}{2}+2\left\| \eta_0\right\|_{B^{\frac{2}{p}+1}_{p,1}}\right)t},
	\end{align*}
	and \begin{align*}
		\|\eta(t)\|_{B^{\frac{2}{p}}_{p,1}}\leq \|\eta_0\|_{B^{\frac{2}{p}}_{p,1}}\e^{Ct\left(\displaystyle\sup_{0\leq s\leq t}\|\eta(s)\|_{B^{\frac{2}{p}+1}_{p,1}}+c_g\right)},
	\end{align*}
where  $c_g$  is a fixed constant characterizing the background stratification.
\end{prop}

\begin{proof}
	
For the sake of simplicity,   we only sketch a priori estimates that can be used to establish local regularity of system \eqref{p-IPM} with initial data  $\eta_0\in B^{\frac{2}{p}+1}_{p,1},$  $2<p<\infty$.
With these estimates, a fully
	rigorous argument can be given in a standard way, using  smooth mollifier
	approximations like in \cite{majda-incompressible flow}.
	
	
	For any integer $k \geq-1$, operating $\Delta_k$ on $\eqref{p-IPM}_1$ yields
	\begin{equation}\label{3.1}
		\partial_t \Delta_k \eta+\left(u\cdot \nabla\right)\Delta_k \eta=-[u\cdot\nabla, {\Delta}_k]\eta+\Delta_k\left(-g'(x_2)\mathcal{R}_{1}^2 \eta\right).	
	\end{equation}
	Taking the $L^p$-norm of \eqref{3.1}, we have
	\begin{equation*}\begin{split}
			\lef \Delta_k\eta \rig_{L^p}\leq \lef \Delta_k\eta_0\rig_{L^p}&+\int_0^t\lef [u\cdot\nabla, {\Delta}_k]\eta\rig_{L^p}\dd \tau\\
			&+\int_0^t\lef \Delta_k\left(-g'(x_2)\mathcal{R}_{1}^2 \eta\right)\rig_{L^p}\dd \tau.
	\end{split}\end{equation*}
	Multiplying  by $2^{k(\frac{2}{p}+1)}$ and summing over $k,$  we deduce from Lemma \ref{lem:commutator} that
	\begin{equation}\label{3.2}
		\begin{split}
			&	\left\| \eta\right\|_{B^{\frac{2}{p}+1}_{p,1}}\\ \leq& \left\| \eta_0\right\|_{B^{\frac{2}{p}+1}_{p,1}}+C\int_0^t \|\nabla u(\tau)\|_{B^{\frac{2}{p}}_{p,1}}\|\eta(\tau)\|_{\Bp}\dd\tau +C\int_0^t \|g'(x_2)\mathcal{R}_1^2\eta\|_{B^{\frac{2}{p}+1}_{p,1}}\dd\tau\\
			\leq& \left\| \eta_0\right\|_{B^{\frac{2}{p}+1}_{p,1}}+C\int_0^t \|\eta(\tau)\|_{\Bp}^2 \dd\tau +Cc_g\int_0^t \|\eta(\tau)\|_{B^{\frac{2}{p}+1}_{p,1}}\dd\tau.
		\end{split}
	\end{equation}
The last inequality follows from the following estimate
\begin{align*}\nonumber
	\|g'(x_2)\mathcal{R}_1^2\eta\|_{B^{\frac{2}{p}+1}_{p,1}}&\leq C\|g'(x_2)\|_{B^{\frac{2}{p}+1}_{\infty,1}(\mr^2)}
	\|\mathcal{R}_1^2\eta\|_{B^{\frac{2}{p}+1}_{p,1}}\\\nonumber
	&\leq C\|g'(x_2)\|_{B^{2}_{\infty,\infty}(\mr^2)}^{\frac{1}{p}+\frac{1}{2}}
	\|g'(x_2)\|_{B^{0}_{\infty,\infty}(\mr^2)}^{\frac{1}{2}-\frac{1}{p}}\|\eta\|_{B^{\frac{2}{p}+1}_{p,1}}\\\nonumber
	&\leq C\|g'(x_2)\|_{W^{2,\infty}(\mr^2)}\|\eta\|_{B^{\frac{2}{p}+1}_{p,1}}\\
	&\leq C\|g'\|_{W^{2,\infty}(\mr)}\|\eta\|_{B^{\frac{2}{p}+1}_{p,1}}\\&\leq Cc_g\|\eta\|_{B^{\frac{2}{p}+1}_{p,1}},
\end{align*}
where we have applied \eqref{product-2-infty}, Lemma \ref{Sobolev} and the embedding properties and the  interpolation inequality in Lemma \ref{lem:besov_properties}.

Therefore,
combining \eqref{3.2} with Young's inequality, we derive	
	\begin{equation*}\begin{split}
		&\sqrt{c_g}\left\| \eta\right\|_{\Bp}^\frac{1}{2}+\left\| \eta\right\|_{\Bp}\\
		&\leq \frac{c_g}{2}+\frac{3}{2}\left\| \eta\right\|_{\Bp}\\
		&\leq \frac{c_g}{2}+2\left\| \eta_0\right\|_{B^{\frac{2}{p}+1}_{p,1}}+C\int_0^t \|\eta\|_{\Bp}^2 \dd\tau +Cc_g\int_0^t \|\eta\|_{B^{\frac{2}{p}+1}_{p,1}}\dd\tau\\
		&\leq \frac{c_g}{2}+2\left\| \eta_0\right\|_{B^{\frac{2}{p}+1}_{p,1}}+C\int_0^t \left(\sqrt{c_g}\left\| \eta\right\|_{\Bp}^\frac{1}{2}+\|\eta\|_{\Bp}\right)^2 \dd\tau. 
\end{split}	
\end{equation*}
	Adopting to Lemma \ref{gronwall} with $f(t)=\sqrt{c_g}\left\| \eta\right\|_{\Bp}^\frac{1}{2}+\left\| \eta\right\|_{\Bp}$ gives to
	\begin{align*}
		\left\| \eta(t)\right\|_{\Bp}&\leq \sqrt{c_g}\left\| \eta\right\|_{\Bp}^\frac{1}{2}+\left\| \eta\right\|_{\Bp}\\
		&\leq \left[\left(\frac{c_g}{2}+2\left\| \eta_0\right\|_{B^{\frac{2}{p}+1}_{p,1}}\right)^{-1}-C\int_{0}^{t}\dd s\right]^{-1}\\
		&=\dfrac{\frac{c_g}{2}+2\left\| \eta_0\right\|_{B^{\frac{2}{p}+1}_{p,1}}}{1-C\left(\frac{c_g}{2}+2\left\| \eta_0\right\|_{B^{\frac{2}{p}+1}_{p,1}}\right)t}.
	\end{align*}
	
	Using  the very  similar way to get \eqref{3.2}, we  estimate the lower-order norm as follows
	\begin{align*}
		\|\eta\|_{B^{\frac{2}{p}}_{p,1}}\leq \|\eta_0\|_{B^{\frac{2}{p}}_{p,1}}+C\sup_{0\leq s\leq t}\|\eta(s)\|_{B^{\frac{2}{p}+1}_{p,1}}\int_{0}^t\|\eta(\tau)\|_{B^{\frac{2}{p}}_{p,1}}\dd\tau+Cc_g\int_{0}^t\|\eta(\tau)\|_{B^{\frac{2}{p}}_{p,1}}\dd\tau.
	\end{align*}
	By applying Lemma \ref{gronwall} again   with $f(t)=\|\eta(t)\|_{B^{\frac{2}{p}}_{p,1}},$ we obatain
	\begin{align*}
		\|\eta(t)\|_{B^{\frac{2}{p}}_{p,1}}\leq \|\eta_0\|_{B^{\frac{2}{p}}_{p,1}}\e^{C t\left(\displaystyle\sup_{0\leq s\leq t}\|\eta(s)\|_{B^{\frac{2}{p}+1}_{p,1}}+c_g\right)}.
	\end{align*}
Hence,   we complete the a priori estimates  of Lemma \ref{local}.
\end{proof}

\section{Proof of Theorem \ref{thm1}}\label{sec:5}

\subsection{Lower bound estimate for $\|\nabla \eta\|_{L^\infty}$}\label{sec:5.1}
To  derive the lower bound estimate for $\|\nabla \eta\|_{L^\infty}$, we apply the orthogonal gradient $\nabla^\perp$ (rather than $\nabla$) to both sides of \eqref{p-IPM}$_1$. This choice is motivated by the structure of the the Biot-Savart law governing the velocity field $u$: since $u$ involves the gradient $\nabla^\perp$, applying $\nabla^\perp$ leads to a more concise  formulation.

Denote  $\Theta=\nabla^\perp \eta$, then $\Theta$ satisfies \begin{align}\label{theta-1}
	\begin{cases}
		\partial_t \Theta+u\cdot \nabla\Theta=-g'(x_2)\mathcal{R}_{1}^2 \Theta-\left(\nabla^\perp g'(x_2)\right)\mathcal{R}_{1}^2 \eta+\Theta\cdot\nabla u,\\[1mm]
		u=	 (-\Delta)^{-\frac12}\mathcal{R}_1\Theta,\\[1mm]
		\Theta({\bf x},0)=\Theta_0=\nabla^\perp\eta_0({\bf x}),
	\end{cases}
\end{align}
where \begin{align*}
	\nabla^\perp g'(x_2)=\begin{pmatrix}
		-g''(x_2)\\
		0
	\end{pmatrix}
.\end{align*}
\begin{prop}\label{lowerbound}
 Let $\eta$ be the  solution of system \eqref{p-IPM} with $\eta_0\in B^{\frac{2}{p}+1}_{p,1},$  $2< p<\infty$. By Lemma \ref{local},  $\eta \in L^\infty(0,T; B^{\frac{2}{p}+1}_{p,1})$ for some $0<T<1$. 
Furthermore, for any $t\in[0,T]$,  the following holds:
\begin{equation}\begin{split}\label{4.6}
		\|\Theta\|_{L^\infty}\geq &t\left\| -g'\mathcal{R}_{1}^2 \Theta_0\right\|_{L^\infty}-
		\left\|\Theta_0\right\|_{L^\infty}-Ct^2\left\| \Theta_0\right\|_{B^{\frac{2}{p}}_{p,1}}\\
	&~~-Ct^2(\|u\|_{L^\infty_tL^\infty}+\|\nabla u\|_{L^\infty_tL^\infty})\left(1+\e^{2t \|\nabla u\|_{L^\infty_tL^\infty}}\right)\|\Theta\|
	_{L^\infty_tB^{\frac{2}{p}}_{p,1}}\\
    &~~-Ct\left(1+\e^{2t\|\nabla u\|_{L^\infty_tL^\infty}}\right)
    \left(\|\Theta\|_{L^\infty_tL^\infty}
    \|\Theta\|_{L^\infty_tB^\frac{2}{p}_{p,1}}
    +\|\eta\|_{L^\infty_tB^\frac{2}{p}_{p,1}}\right),
\end{split}\end{equation}
where  $C$ are universal constants depending  on $g(x_2)$,  but independent of time $t$.
\end{prop}

\begin{proof}
	Consider the Lagrangian flow-map $\Phi $ being transported by the velocity $u$,
	\begin{align*}
		\begin{cases}
			\partial_t\Phi({\bf x},t)=u(\Phi({\bf x},t),t),\\
			\Phi({\bf x},0)={\bf x}.
		\end{cases}
	\end{align*}
	Compounding the Lagrangian flow-map $\Phi$ on both sides of $\eqref{theta-1}_1$  yields
\begin{align*}
\left(	\partial_t \Theta+u\cdot \nabla\Theta\right)\circ \Phi=\left(\Theta\cdot\nabla u-g'(x_2) \mathcal{R}_{1}^2 \Theta-\left(\nabla^\perp g'(x_2) \right)\mathcal{R}_{1}^2 \eta\right)\circ \Phi.
\end{align*}
Then,  $\Theta\circ \Phi$ satisfies
	\begin{align*}
			\partial_t \left(\Theta\circ \Phi\right)&+g'(x_2)\mathcal{R}_{1}^2 \left(\Theta\circ\Phi\right)\\
			&=\left(-g' \circ\Phi\right)\left[\mathcal{R}_{1}^2 ,\Phi\right]\Theta+\left(g'-g' \circ\Phi\right)\mathcal{R}_{1}^2 \left(\Theta\circ\Phi\right)\\
		&~~~~~~+\left(\Theta\cdot\nabla u-\left(\nabla^\perp g' \right)\mathcal{R}_{1}^2 \eta\right)\circ \Phi
	\end{align*}
	with
	\begin{align*}
		\left[\mathcal{R}_{1}^2 ,\Phi\right]\Theta=\mathcal{R}_{1}^2 (\Theta)\circ\Phi-\mathcal{R}_{1}^2 \left(\Theta\circ\Phi\right).
	\end{align*}
By Duhamel's principle,
	\begin{align}\nonumber
		\Theta\circ \Phi=\e^{-g'\mathcal{R}_{1}^2 t}\Theta_0&+\int_{0}^t\e^{-g'\mathcal{R}_{1}^2 (t-\tau)}\left(-g' \circ\Phi\right)\left[\mathcal{R}_{1}^2 ,\Phi\right]\Theta(\tau)\dd\tau\\\nonumber
		&+\int_{0}^t\e^{-g'\mathcal{R}_{1}^2 (t-\tau)}\left(\left(g'-g' \circ\Phi\right)\mathcal{R}_{1}^2 \left(\Theta\circ\Phi\right)\right)(\tau)\dd\tau\\\label{4.2}
	&+\int_{0}^t\e^{-g'\mathcal{R}_{1}^2 (t-\tau)}\left(\left(\Theta\cdot\nabla u-\left(\nabla^\perp g' \right)\mathcal{R}_{1}^2 \eta\right)\circ \Phi\right)(\tau)\dd\tau,
	\end{align}  where  the operator  $\e^{-g'\mathcal{R}_{1}^2 t}$ can be written as follows by the Taylor series expansion,
\begin{align}\label{4.3}
\e^{-g'\mathcal{R}_{1}^2 t}=I+(-g'\mathcal{R}_{1}^2 t)+t^2\sum_{n=2}^{\infty}\dfrac{t^{n-2}\left(-g'\mathcal{R}_{1}^2 \right)^n}{n!}.
\end{align}
Then,
\begin{equation*}
\begin{split}
	\|\Theta\|_{L^\infty}&\geq \lef\e^{-g'\mathcal{R}_{1}^2 t}\Theta_0\rig_{L^\infty}-\lef\int_{0}^t\e^{-g'\mathcal{R}_{1}^2 (t-\tau)}\left(-g' \circ\Phi\right)\left[\mathcal{R}_{1}^2 ,\Phi\right]\Theta(\tau)\dd\tau\rig_{L^\infty}\\\nonumber
	&~~~-\left\|\int_{0}^t\e^{-g'\mathcal{R}_{1}^2 (t-\tau)}\left(\left(g'-g' \circ\Phi\right)\mathcal{R}_{1}^2 \left(\Theta\circ\Phi\right)\right)(\tau)\dd\tau\right\|_{L^\infty}\\
	&~~~-\left\|\int_{0}^t\e^{-g'\mathcal{R}_{1}^2 (t-\tau)}\left(\left(\Theta\cdot\nabla u-\left(\nabla^\perp g' \right)\mathcal{R}_{1}^2 \eta\right)\circ \Phi\right)(\tau)\dd\tau\right\|_{L^\infty}\\\label{4.5}
	&\define \mathfrak{B_1}-\mathfrak{B_2}-\mathfrak{B_3}-\mathfrak{B_4}.
\end{split}\end{equation*}

For $\mathfrak{B_1}$, using \eqref{4.3} gives to
	\begin{align*}
		\mathfrak{B_1}
		&\geq \left\| -g'\mathcal{R}_{1}^2 t\Theta_0\right\|_{L^\infty}-\left\|\Theta_0\right\|_{L^\infty}-\left\|t^2\sum_{n=2}^{\infty}\dfrac{t^{n-2}\left(-g'\mathcal{R}_{1}^2 \right)^n}{n!}\Theta_0\right\|_{L^\infty}.
	\end{align*}
Noting that  the Riesz operator $\mathcal{R}_{1}$ is bounded in $B^{\frac{2}{p}}_{p,1}(\mr^2)$ with $2< p<\infty$, we can get for any $0\leq t\leq T<1,$
	\begin{align*}
		\left\|t^2\sum_{n=2}^{\infty}\dfrac{t^{n-2}\left(-g'\mathcal{R}_{1}^2 \right)^n}{n!}\Theta_0\right\|_{L^\infty}
       &\leq t^2\sum_{n=2}^\infty\dfrac{1 }{n!}\|\left(-g'\right)^n\|_{L^\infty}\|\mathcal{R}_{1}^{2n}\Theta_0\|_{L^\infty}
       \\
	   &\leq 
	   C t^2\sum_{n=2}^\infty \dfrac{c_g^n }{n!}\|\Theta_0\|_{B^\frac{2}{p}_{p,1}}\\
	   &\leq Ct^2\left\| \Theta_0\right\|_{B^\frac{2}{p}_{p,1}}.
	\end{align*}
	As a result, we obtain
	\begin{align*}
	\mathfrak{B_1}\geq t\left\| -g'\mathcal{R}_{1}^2 \Theta_0\right\|_{L^\infty}-\left\|\Theta_0\right\|_{L^\infty}-C_1t^2\left\| \Theta_0\right\|_{B^{\frac{2}{p}}_{p,1}}.
	\end{align*}

For $\mathfrak{B_2}$, using \eqref{4.3}, \eqref{product-2-infty}  and  the continuous embedding  $B^{\frac{2}{p}}_{\infty,1}\hookrightarrow L^\infty$ yields
\begin{equation}\label{B2}
		\begin{split}
		\mathfrak{B_2}&=\left\|\int_{0}^t\e^{-g'\mathcal{R}_{1}^2 (t-\tau)}\left(-g' \circ\Phi\right)\left[\mathcal{R}_{1}^2 ,\Phi\right]\Theta(\tau)\dd\tau\right\|_{L^\infty}\\
				&\leq C\sum_{n=0}^\infty \dfrac{c_g^n }{n!} \int_0^t \left\|\left(-g' \circ\Phi\right)\left[\mathcal{R}_{1}^2 ,\Phi\right]\Theta(\tau)\right\|_{B^{\frac{2}{p}}_{p,1}}\dd \tau\\
				&\leq C \int_0^t \|g' \circ\Phi\|_{B^{\frac{2}{p}}_{\infty,1}}\left\|\left[\mathcal{R}_{1}^2 ,\Phi\right]\Theta(\tau)\right\|_{B^{\frac{2}{p}}_{p,1}}\dd \tau.
	\end{split}\end{equation}
By  the  interpolation inequality in Lemma \ref{lem:besov_properties},
  Lemma \ref{lagrange-flowmap} and the continuous embedding  $W^{1,\infty}\hookrightarrow B^{1}_{\infty, \infty}$, we have
\begin{equation}\begin{split}\label{g-infty}
	\|g' \circ\Phi\|_{B^{\frac{2}{p}}_{\infty,1}}&\leq C\|g' \circ\Phi\|_{B^{1}_{\infty,\infty}}^{\frac{2}{p}}\|g' \circ\Phi\|_{B^{0}_{\infty,\infty}}^{1-\frac{2}{p}}\\
	&\leq C\|g'\circ \Phi\|_{W^{1,\infty}(\mr^2)}\\
	&\leq C\|g'\|_{W^{2,\infty}(\mr)}\|\nabla \Phi\|_{L^\infty}\leq C\e^{t \|\nabla u\|_{L^\infty_tL^\infty}}.
\end{split}\end{equation}
In addition, by Lemma \ref{lagrange-flowmap} and Lemma \ref{commutator-elgindi} (applied with $0<\alpha=\frac{2}{p}<1$), we have
\begin{equation*} 
	\begin{split}
\left\|\left[\mathcal{R}_{1}^2 ,\Phi\right]\Theta(\tau)\right\|_{B^{\frac{2}{p}}_{p,1}}&\leq Ct\|\nabla u\|_{L^\infty_tL^\infty}\e^{t \|\nabla u\|_{L^\infty_tL^\infty}}\|\Theta\|_{B^{\frac{2}{p}}_{p,1}}.
\end{split} 
\end{equation*}
Therefore,	\begin{align*}
	\mathfrak{B_2}\leq  Ct^2\|\nabla u\|_{L^\infty_tL^\infty}\e^{2t \|\nabla u\|_{L^\infty_tL^\infty}}\|\Theta\|_{L^\infty_tB^{\frac{2}{p}}_{p,1}}.
\end{align*}

For $\mathfrak{B_3}$, following similar arguments used to establish \eqref{B2} and \eqref{g-infty}, we obtain
\begin{equation*}\begin{split}
\mathfrak{B_3}&=\left\|\int_{0}^t\e^{-g'\mathcal{R}_{1}^2 (t-\tau)}\left(\left(g'-g' \circ\Phi\right)\mathcal{R}_{1}^2 \left(\Theta\circ\Phi\right)\right)(\tau)\dd\tau\right\|_{L^\infty}\\
&\leq C\int_{0}^t\left\|\left(\left(g'-g' \circ\Phi\right)\mathcal{R}_{1}^2 \left(\Theta\circ\Phi\right)\right)(\tau)\right\|_{B^{\frac{2}{p}}_{p,1}}\dd\tau\\
&\leq Ct\|g'-g'\circ \Phi\|_{L^\infty_tB^\frac{2}{p}_{\infty,1}}\|\mathcal{R}_{1}^2 \left(\Theta\circ\Phi\right)\|_{L^\infty_tB^\frac{2}{p}_{p,1}}\\
&\leq Ct\|g'-g'\circ \Phi\|_{L^\infty_tW^{1,\infty}}\| \Theta\circ\Phi\|_{L^\infty_tB^\frac{2}{p}_{p,1}}.
\end{split}
\end{equation*}
By  exploiting the measure-preserving property of $\Phi$ along with  Lemmas \ref{lagrange-flowmap} and \ref{phi-composite} (applied with $\alpha=\frac{2}{p}$), we have 
\begin{equation*}\begin{split}
	&\| \Theta\circ\Phi\|_{L^\infty_tB^\frac{2}{p}_{p,1}}\\
	\leq& C\|\Theta\|_{L^\infty_t L^p}+\|\nabla \Phi\|_{L^\infty}^\frac{2}{p}\|\Theta\|_{L^\infty_t\dot{B}^\frac{2}{p}_{p,1}}\\
	\leq& C\left(1+\e^{\frac{2}{p}t\|\nabla u\|_{L^\infty_tL^\infty}}\right)\|\Theta\|_{L^\infty_tB^{\frac{2}{p}}_{p,1}}.
\end{split}\end{equation*}
Furthermore, applying the mean-value theorem yields
\begin{equation*}\begin{split}
	&\|g'-g'\circ \Phi\|_{L^\infty_t	W^{1,\infty}}\\
		\leq& C\|g'\|_{W^{2,\infty}(\mr)}\|\Phi_2-x_2\|_{L^\infty_t W^{1,\infty}}\\
		\leq &Ct(\| u\|_{L^\infty_t L^\infty}+\|\nabla u\|_{L^\infty_t L^\infty}\|\nabla \Phi\|_{L^\infty_t L^\infty}).
\end{split}\end{equation*}
Therefore, $\mathfrak{B_3}$ can be established as follows
\begin{equation*}\begin{split}
		\mathfrak{B_3}& \leq Ct^2\left(1+\e^{\frac{2}{p}t\|\nabla u\|_{L^\infty_tL^\infty}}\right)(\| u\|_{L^\infty_t L^\infty}+\|\nabla u\|_{L^\infty_t L^\infty}\e^{t\|\nabla u\|_{L^\infty_tL^\infty}})\|\Theta\|_{L^\infty_tB^{\frac{2}{p}}_{p,1}}.
	\end{split}
\end{equation*}

For $\mathfrak{B_4}$, we  deduce that 
\begin{equation}	\begin{split}\label{4.4}
			\mathfrak{B_4}&=\left\|\int_{0}^t\e^{-g'\mathcal{R}_{1}^2 (t-\tau)}\left(\left(\Theta\cdot\nabla u-\left(\nabla^\perp g' \right)\mathcal{R}_{1}^2 \eta\right)\circ \Phi\right)(\tau)\dd\tau\right\|_{L^\infty}\\
			&\leq  C\int_{0}^t\left\|\left(\left(\Theta\cdot\nabla u-\left(\nabla^\perp g' \right)\mathcal{R}_{1}^2 \eta\right)\circ \Phi\right)(\tau)\right\|_{B^\frac{2}{p}_{p,1}}\dd\tau\\
				&\leq C\int_{0}^t\left\|\left(\left(\Theta\cdot\nabla u-\left(\nabla^\perp g' \right)\mathcal{R}_{1}^2 \eta\right)\circ \Phi\right)(\tau)\right\|_{L^p}\dd\tau\\
			&~~~~~~~+C\int_{0}^t\left\| \left(\left(\Theta\cdot\nabla u-\left(\nabla^\perp g' \right)\mathcal{R}_{1}^2 \eta\right)\circ \Phi\right)(\tau)\right\|_{\dot{B}^\frac{2}{p}_{p,1}}\dd\tau\\&=:\mathfrak{B_{41}}+\mathfrak{B_{42}}.
	\end{split}\end{equation}
	Since the flow map $\Phi$ is measure-preserving and  the velocity gradient satisfies $\nabla u=\mathcal{R}\mathcal{R}_1 \Theta,$  an application of the  H\"{o}lder inequality yields
\begin{align*}\nonumber
	&\left\|\left(\left(\Theta\cdot\nabla u-\left(\nabla^\perp g' \right)\mathcal{R}_{1}^2 \eta\right)\circ \Phi\right)(\tau)\right\|_{L^p}\\
	=&\left\|\left(\Theta\cdot\nabla u-\left(\nabla^\perp g' \right)\mathcal{R}_{1}^2 \eta\right)(\tau)\right\|_{L^p}\\
	\leq & C\left(\|\Theta(\tau)\|_{L^\infty}\|\nabla u(\tau)\|_{L^p}+	\left\| \eta(\tau)\right\|_{L^p}\left\|\nabla^\perp g' \right\|_{L^\infty(\mr^2)}\right)\\
	\leq& C\left(\|\Theta(\tau)\|_{L^\infty}\|\Theta(\tau)\|_{L^p}+	\left\| \eta(\tau)\right\|_{L^p}\left\|g'\right\|_{W^{1,\infty}(\mr)}\right).
\end{align*}Then, we get
\begin{equation*}\begin{split}
		\mathfrak{B_{41}}&\leq  Ct\left(\|\Theta\|_{L^\infty_tL^\infty}\|\Theta\|_{L^\infty_tL^p}+\|\eta\|_{L^\infty_tB^\frac{2}{p}_{p,1}}\right).
\end{split}\end{equation*}
In view of the estimate  \eqref{product-1-riesz} of Lemma \ref{riesz-product}, it follows that
\begin{equation}\begin{split}\label{B4-homo-1}
	\left\|\left(\Theta\cdot\nabla u\right)(\tau)\right\|_{\dot{B}^\frac{2}{p}_{p,1}}&=\|\Theta\cdot \mathcal{R}\mathcal{R}_1 \Theta(\tau)\|_{\dot{B}^\frac{2}{p}_{p,1}}\\
	&\leq C\|\Theta(\tau)\|_{L^\infty}\|\Theta(\tau)\|_{B^{\frac{2}{p}}_{p,1}}.
\end{split}\end{equation}
Using  \eqref{product-2-infty} and Lemma \ref{lem:besov_properties} again  gives to
\begin{equation}\begin{split}\label{B4-homo-2}
	\|\nabla^\perp g' \mathcal{R}_{1}^2 \eta(\tau)\|_{\dot{B}^\frac{2}{p}_{p,1}}&\leq\| g''(x_2) \mathcal{R}_{1}^2 \eta(\tau)\|_{B^\frac{2}{p}_{p,1}}\\
	&\leq C\|g''(x_2)\|_{B^{\frac{2}{p}}_{\infty,1}(\mr^2)}\|\mathcal{R}_{1}^2 \eta(\tau)\|_{B^\frac{2}{p}_{p,1}}\\
	&\leq C\|g'\|_{W^{2,\infty}(\mr)}\|\eta(\tau)\|_{B^\frac{2}{p}_{p,1}}\leq C\|\eta(\tau)\|_{B^\frac{2}{p}_{p,1}}.
\end{split}\end{equation}
By combining the estimates  \eqref{B4-homo-1} and \eqref{B4-homo-2} with Lemma \ref{phi-composite}, we obtain \begin{equation*}\begin{split}\mathfrak{B_{42}}&\leq
	C	\int_{0}^t\|\nabla \Phi (\tau) \|_{L^\infty}^\frac{2}{p}\left(\left\|\left(\Theta\cdot\nabla u\right)(\tau)\right\|_{\dot{B}^\frac{2}{p}_{p,1}}+\|\nabla^\perp g' \mathcal{R}_{1}^2 \eta(\tau)\|_{\dot{B}^\frac{2}{p}_{p,1}}\right)\dd\tau
		\\&\leq Ct\e^{\frac{2}{p}t\|\nabla u\|_{L^\infty_tL^\infty}}\left(\|\Theta\|_{L^\infty_tL^\infty}\|\Theta\|_{L^\infty_tB^\frac{2}{p}_{p,1}}+\|\eta\|_{L^\infty_tB^\frac{2}{p}_{p,1}}\right).
\end{split}\end{equation*}
Therefore, we ultimately obtain
	\begin{align*}
		\mathfrak{B_4}&\leq Ct\left(1+\e^{\frac{2}{p}t\|\nabla u\|_{L^\infty_tL^\infty}}\right)\left(\|\Theta\|_{L^\infty_tL^\infty}\|\Theta\|_{L^\infty_tB^\frac{2}{p}_{p,1}}+\|\eta\|_{L^\infty_tB^\frac{2}{p}_{p,1}}\right).
\end{align*} 
By incorporating the estimates for 	 $\mathfrak{B_1}$-$\mathfrak{B_4}$  into \eqref{4.5}, we derive a lower bound estimate for $\|\Theta\|_{L^\infty}$:
	\begin{equation*}\begin{split}
			\|\Theta\|_{L^\infty}\geq &t\left\| -g'\mathcal{R}_{1}^2 \Theta_0\right\|_{L^\infty}-
			\left\|\Theta_0\right\|_{L^\infty}-Ct^2\left\| \Theta_0\right\|_{B^{\frac{2}{p}}_{p,1}}\\
			&~~-Ct^2(\|u\|_{L^\infty_tL^\infty}+\|\nabla u\|_{L^\infty_tL^\infty})\left(1+\e^{2t \|\nabla u\|_{L^\infty_tL^\infty}}\right)\|\Theta\|
			_{L^\infty_tB^{\frac{2}{p}}_{p,1}}\\
			&~~-Ct\left(1+\e^{2t\|\nabla u\|_{L^\infty_tL^\infty}}\right)
			\left(\|\Theta\|_{L^\infty_tL^\infty}
			\|\Theta\|_{L^\infty_tB^\frac{2}{p}_{p,1}}
			+\|\eta\|_{L^\infty_tB^\frac{2}{p}_{p,1}}\right).
	\end{split}\end{equation*}
This establishes \eqref{4.6} and consequently completes the proof of the lemma. 
\end{proof}

\subsection{Conclusion of the Proof}\label{sec:5.2}
\begin{proof}[Proof of Theorem \ref{thm1}]
	
	
	We initialize the perturbation as
	\begin{align*}
	\eta_0({\bf x})=\rho_0({\bf x})-g(x_2)=\dfrac{f_N}{\sqrt{N}}\in C_c^\infty(\mr^2),
	\end{align*}
	where $f_N$ is constructed via Proposition \ref{fN-sequence}. Given that $g'(x_2) \not\equiv 0$, we choose $h(\mathbf{x}) = g'(x_2)$ (depending only on $x_2$) in accordance with Proposition \ref{fN-sequence}, thereby obtaining the crucial estimates
	\begin{align}\label{5.3}
		\|\eta_0(\cdot)\|_{W^{1,\infty}}\leq \dfrac{c}{\sqrt{N}},&~\|\eta_0(\cdot)\|_{B^{\frac{2}{p}}_{p,1}}\leq  \dfrac{c\left(\log N\right)^{\frac{1}{2}}}{\sqrt{N}}, ~\|\eta_0(\cdot)\|_{B^{\frac{2}{p}+1}_{p,1}}\leq c\sqrt{N},\\\label{5.4}
		&\|g'\mathcal{R}_1^2\nabla^\perp \eta_0\|_{L^\infty}\geq c_1\sqrt{N},
	\end{align}
where the constants $c,c_1>0$ are independent of $N$.

By applying Proposition \ref{local} together with \eqref{5.3}, we obtain, for sufficiently small $t>0$,
\begin{align*}
	\lef \eta(t)\rig_{B^{\frac{2}{p}+1}_{p,1}}\leq \dfrac{\frac{c_g}{2}+2\left\| \eta_0\right\|_{B^{\frac{2}{p}+1}_{p,1}}}{1-C\left(\frac{c_g}{2}+2\left\| \eta_0\right\|_{B^{\frac{2}{p}+1}_{p,1}}\right)t}\leq c_2\sqrt{N}+c_3,
\end{align*}
and
\begin{equation*}\begin{split}
	\|\eta(t)\|_{B^{\frac{2}{p}}_{p,1}}&\leq \|\eta_0\|_{B^{\frac{2}{p}}_{p,1}}\e^{Ct\left(\displaystyle\sup_{0\leq s\leq t}\|\eta(s)\|_{B^{\frac{2}{p}+1}a_{p,1}}+c_g\right)}\\ &\leq \frac{\left(\log N\right)^{\frac{1}{2}}}{\sqrt{N}}\e^{t(c_4\sqrt{N}+c_5)}.\end{split}
\end{equation*}
The continuous embedding $B^{\frac{2}{p}}_{p,1} \hookrightarrow L^\infty$ and $L^p$- boundedness of the Riesz operator gives, for all $t>0$,
\begin{align*}
	&	\| u(t)\|_{L^\infty},~	\|\nabla u(t)\|_{L^\infty}\leq\|u(t)\|_{B^{\frac{2}{p}+1}_{p,1}}\leq C\|\eta(t)\|_{B^{\frac{2}{p}+1}_{p,1}} \leq c_2\sqrt{N}+c_3.\\
	&	\|\Theta(t)\|_{B^\frac{2}{p}_{p,1}}\leq C\|\eta(t)\|_{B^{\frac{2}{p}+1}_{p,1}} \leq c_2\sqrt{N}+c_3.
\end{align*}
On the other hand,  by  \eqref{5.4},  we have
\begin{align*}
	\left\| -g'\mathcal{R}_{1}^2 \Theta_0\right\|_{L^\infty}\geq c_1\sqrt{N}.
\end{align*}

Now, substituting these estimates into \eqref{4.6}, we obtain the lower bound:
\begin{align}\label{eq:theta-lower-bound}
	\|\Theta(t)\|_{L^\infty} &\geq c_1t\sqrt{N} - \frac{c}{\sqrt{N}} - ct^2\sqrt{N} \nonumber \\
	&\quad - t^2(c_2\sqrt{N}+c_3)\left(1+e^{2t (c_2\sqrt{N}+c_3)}\right)(c_2\sqrt{N}+c_3) \nonumber \\
	&\quad - t\left(1+e^{2t(c_2\sqrt{N}+c_3)}\right) \nonumber \\
	&\quad \times \left(\|\Theta\|_{L^\infty_tL^\infty}(c_2\sqrt{N}+c_3) + \frac{(\log N)^{1/2}}{\sqrt{N}}e^{t(c_4\sqrt{N}+c_5)}\right).
\end{align}
We select the time scale parameter as
$$t=\frac{M}{\sqrt{N}},$$ where the constant $M>0$  will be determined later. Under this choice, we observe that
 the exponential terms $e^{2t (c_2\sqrt{N}+c_3)}$ and $e^{t(c_4\sqrt{N}+c_5)}$ remain uniformly bounded in $N$.
Consequently, for sufficiently large  $N,$ the lower bound estimate \eqref{eq:theta-lower-bound} reduces to the simplified form
\begin{align}\label{eq:simplified-bound}
	\|\Theta\|_{L^\infty} &\geq c_1t\sqrt{N} - \frac{c}{\sqrt{N}} - ct^2\sqrt{N} - \bar{c}_2t^2N \nonumber \\
	&\quad - \bar{c}_3t\sqrt{N}\|\Theta\|_{L^\infty_tL^\infty} - ct\frac{(\log N)^{1/2}}{\sqrt{N}} \nonumber \\
	&\geq  \bar{c}_1t\sqrt{N} -  \bar{c}_2t^2N -  \bar{c}_3t\sqrt{N}\|\Theta\|_{L^\infty_tL^\infty}.
\end{align}

Without loss of generality, we may assume $\|\Theta\|_{L^\infty_tL^\infty} \leq \dfrac{\bar{c}_1}{2\bar{c}_3}$, since otherwise the desired result follows immediately.  Under this assumption, the simplified bound \eqref{eq:simplified-bound} takes the form:
\begin{align*}
	\|\Theta(t)\|_{L^\infty} \geq \frac{\bar{c}_1}{2}t\sqrt{N} - \bar{c}_2t^2N.
\end{align*}
As a result,  by selecting  $M =t\sqrt{N}= \frac{\bar{c}_1}{4\bar{c}_2},$ we obtain \begin{equation*}\begin{split}
	\sup_{0\leq t\leq \frac{\bar{c}_1}{4\bar{c}_2\sqrt{N}}} \|\Theta(t)\|_{L^\infty} &\geq|\Theta(\frac{\bar{c}_1}{4\bar{c}_2\sqrt{N}})|\\ &\geq \frac{\bar{c}_{1}^2}{16\bar{c}_2} > 0.
\end{split}\end{equation*}
The proof is completed by taking $\varepsilon = \dfrac{1}{\sqrt{N}},$  $\tilde{c} = \dfrac{\bar{c}_1}{4\bar{c}_2}$ and $\tilde{C} = \dfrac{\bar{c}_1}{16\bar{c}_2}$.
\end{proof}

\vspace{2mm}
{\bf Conflict of interest:} 
The authors declare that they have no conflict of interest.
\vspace{2mm}

\section*{Acknowledgments}
Yu was partially supported by the	National Natural Science Foundation of China (NNSFC) (No.11901040), Beijing Natural Science Foundation (BNSF) (No.1204030).

\appendix
\section{Appendix}\label{sec:appendixA}

 In this section, we provide the definitions of the homogeneous and the inhomogeneous Besov spaces. They are defined through the Littlewood-Paley decomposition. Several related embedding relations and inequalities will also be given here.
Most of the materials in this section are classical and
we refer the reader to the book \cite[Chapter 2]{bahouriFourierAnalysisNonlinear2011} for more details.

\subsection{Littlewood-Paley Decomposition}
Consider the annulus $\mathcal{C} := \{\xi \in \mathbb{R}^n : \frac{3}{4} \leq |\xi| \leq \frac{8}{3}\}$. We can construct radial functions  $\chi \in C_c^\infty(B(0,\frac{4}{3}))$ and $\varphi \in C_c^\infty(\mathcal{C})$, valued in $[0,1]$, that  form a partition of unity satisfying
\begin{equation*}
	\sum_{j \in \mathbb{Z}} \varphi(2^{-j}\xi) = 1 \quad \forall \xi \in \mathbb{R}^n \setminus \{0\}, 
\end{equation*}
\begin{equation*}
	\chi(\xi) + \sum_{j \geq 0} \varphi(2^{-j}\xi) = 1 \quad \forall \xi \in \mathbb{R}^n,
\end{equation*}
\begin{equation*}
suup\varphi(2^{-j}\cdot)\cap suup\varphi(2^{-j'}\cdot)=\emptyset,~~|j-j'|\geq2,
\end{equation*}
\begin{equation*}
suup\chi(\cdot)\cap suup\varphi(2^{-j}\cdot)=\emptyset,~~j \geq1.
\end{equation*}

The nonhomogeneous and homogeneous dyadic blocks are defined as follows. For the nonhomogeneous case:
\begin{align*}
	\Delta_j u &\define
	\begin{cases}
		\varphi(2^{-j}D)u = 2^{jn} \int_{\mathbb{R}^n} h(2^j y) u(x-y) dy & \text{if } j \geq 0, \\[1mm]
		\chi(D)u = \int_{\mathbb{R}^n} \tilde{h}(y) u(x-y) dy & \text{if } j = -1, \\[1mm]
		0 & \text{if } j \leq -2,
	\end{cases}
\end{align*}
\begin{equation*} \mathcal{S}_j u =\chi(2^{-j}D)u=2^{jn} \int_{\mathbb{R}^n} \tilde{h}(2^j y) u(x-y) dy,~~j \geq 0,
\end{equation*}
and for the homogeneous case:
\begin{equation*}
\dot{\Delta}_j u \define \varphi(2^{-j}D)u = 2^{jn} \int_{\mathbb{R}^n} h(2^j y) u(x-y) dy \quad \text{for}~ j \in \mathbb{Z},
\end{equation*}
\begin{equation*}\dot{\mathcal{S}}_j u=\chi(2^{-j}D)u,~~j \in \mathbb{Z}.
\end{equation*}
Here $h = \mathscr{F}^{-1}\varphi$ and $\tilde{h} = \mathscr{F}^{-1}\chi$ denote the inverse Fourier transforms of the chosen smooth functions $\varphi$ and $\chi$ respectively.
It is straightforward to derive that the low-frequency cut-off operators $\mathcal{S}_j (\dot{\mathcal{S}}_j)$ are also given by
\begin{align*}
	\mathcal{S}_j u &= \sum_{\ell \leq j-1} \Delta_\ell u,  ~\forall u \in \mathcal{S}'(\mathbb{R}^n), \\
	\dot{\mathcal{S}}_j u &= \sum_{\ell \leq j-1} \dot{\Delta}_\ell u,~ \forall u \in \mathcal{S}_h'(\mathbb{R}^n) := \mathcal{S}'(\mathbb{R}^n)/\wp(\mathbb{R}^n),
\end{align*} where $\wp(\mathbb{R}^n)$ denotes the space of polynomials.
The operators $\Delta_j (\dot{\Delta}_j),$ $\mathcal{S}_j (\dot{\mathcal{S}}_j)$ are $(p,p)$-bounded uniformly in 
$j.$ This fact will be consistently used throughout our analysis.


\subsection{Besov Spaces}
\begin{definition}[Homogeneous Besov Space]
	For $s \in \mathbb{R}$ and $1 \leq p, r \leq \infty$, the homogeneous Besov space $\dot{B}^s_{p,r}$ consists of all tempered distributions $u \in \mathcal{S}_h'(\mathbb{R}^n)$ such that
	\begin{equation*}
		\|u\|_{\dot{B}^s_{p,r}} := \left\| \left(2^{js} \|\dot{\Delta}_j u\|_{L^p}\right)_{j \in \mathbb{Z}} \right\|_{\ell^r(\mathbb{Z})} < \infty.
	\end{equation*}
\end{definition}

\begin{definition}[Nonhomogeneous Besov Space]
	For $s \in \mathbb{R}$ and $1 \leq p, r \leq \infty$, the nonhomogeneous Besov space $B^s_{p,r}$ consists of all tempered distributions $u \in \mathcal{S}'(\mathbb{R}^n)$ such that
	\begin{equation*}
		\|u\|_{B^s_{p,r}} := \left\| \left(2^{js} \|\Delta_j u\|_{L^p}\right)_{j \in \mathbb{Z}} \right\|_{\ell^r(\mathbb{Z})} < \infty.
	\end{equation*}
\end{definition}

Many frequently used function spaces are special cases of Besov spaces. The
following lemma lists some useful equivalence and embedding relations.

\begin{lemma}\label{Sobolev}
For all $s\in \mr,$	 $H^s \sim B^s_{2,2},$  $\dot{H}^s\sim \dot{B}^s_{2,2}$.
For all $s\in \mr,$ $W^{s,\infty}\hookrightarrow B^{s}_{\infty,\infty}.$
\end{lemma}

The Besov spaces defined above obey various inclusion relations. In particular, we have the following lemmas.
\begin{lemma}\label{lem:besov_properties}
Assume that $s\in\mr$ and $p, r\in[1,\infty].$

(1) If $s > 0$, then
\begin{equation*}
	B^{s}_{p,r} \hookrightarrow \dot{B}^{s}_{p,r},~~ p < \infty, \end{equation*}
\begin{equation*}	B^{s}_{\infty,r} \cap \mathcal{S}'_h \hookrightarrow \dot{B}^{s}_{\infty,r}.
\end{equation*}

(2) If $s_1\leq s_2,$ then $B^{s_2}_{p,r} \hookrightarrow B^{s_1}_{p,r}.$
This inclusion relation is false for the homogeneous Besov spaces.

(3)	\textbf{Besov embedding properties}: If $1\leq p\leq \widetilde{p}\leq\infty,$  $1\leq r\leq \widetilde{r}\leq\infty$, then, for any $s\in\mathbb{R}$,
		\begin{equation*}\begin{split}
			B^{s}_{p,r} &\hookrightarrow B^{s-n(1/p-1/\widetilde{p})}_{\widetilde{p},\widetilde{r}}, \\
			B^{n/p}_{p,1} &\hookrightarrow L^\infty. 
	\end{split}	\end{equation*}
		
(4) \textbf{Interpolation inequalities}: Assume that $s_1 < s_2$, $1 \leq p \leq \infty$, $1 \leq r \leq \infty$ and $0 < \theta < 1$, then
		\begin{align*}
			\|u\|_{B^{\theta s_1+(1-\theta)s_2}_{p,r}} &\leq \|u\|^\theta_{B^{s_1}_{p,r}} \|u\|^{1-\theta}_{\dot{B}^{s_2}_{p,r}}, \\
			\|u\|_{B^{\theta s_1+(1-\theta)s_2}_{p,1}} &\leq \frac{C}{s_2-s_1}\left(\frac{1}{\theta}+\frac{1}{1-\theta}\right) \|u\|^\theta_{B^{s_1}_{p,\infty}} \|u\|^{1-\theta}_{B^{s_2}_{p,\infty}}.
		\end{align*}
		

\end{lemma}

\begin{lemma}\label{lem:besov_compact}
	Let $K \subset \mathbb{R}^n$ be compact. For $s > 0$, the spaces $B^s_{p,r}(K)$ and $\dot{B}^s_{p,r}(K)$ coincide. Moreover, there exists $C > 0$ such that for any $f \in \dot{B}^s_{p,r}(K)$,
	\begin{align*}
		\|f\|_{B^s_{p,r}(K)} \leq C(1 + |K|)^{s/n} \|f\|_{\dot{B}^s_{p,r}(K)}.
	\end{align*}
\end{lemma}


	

\subsection{Nonhomogeneous Bony Decomposition}
In this subsection, we recall nonhomogeneous Bony decomposition:
\begin{equation*}
	uv = T_u v + T_v u + R(u,v),
\end{equation*}
with nonhomogeneous paraproduct of $v$ by $u$
\begin{equation*}
	T_u v := \sum_{j \geq -1} S_{j-1} u \Delta_j v, 
\end{equation*} and nonhomogeneous remainder of $u$ and $v$
\begin{equation*}
	R(u,v) :=  \sum_{|j-k| \leq 1} \Delta_{j} u \Delta_k v.
\end{equation*}

\begin{lemma}\label{B}
There exist constants $C>0$, such that for any $s,s_1,s_2\in\mr$, and $1\leq p, p_1, p_2, r, r_1, r_2\leq \infty,$
 \begin{equation*}
	\|T_u v\|_{B^s_{p,r}}\leq C^{|s|+1}\|u\|_{L^\infty}\|v\|_{B^s_{p,r}},
\end{equation*} 
\begin{equation*}
\|R_u v\|_{B^{\widetilde{s}}_{\widetilde{p},r}}\leq \frac{C^{|s_1+s_2|+1}}{s_1+s_2}\|u\|_{B^{s_1}_{p_1,r_1}}\|v\|_{B^{s_2}_{p_2,r_2}},
\end{equation*} where $s_1+s_2>0,$ $\frac{1}{\widetilde{p}}\leq\frac{1}{p_1}+\frac{1}{p_2}\leq 1,$ $\frac{1}{r_1}+\frac{1}{r_2}=\frac{1}{r}\leq1,$ $\widetilde{s}-\frac{n}{p}=s_1-\frac{n}{p_1}+s_2-\frac{n}{p_2}.$
\end{lemma}



\subsection{Commutator  Estimates}
The following estimates are crucial for handling nonlinear terms in our analysis of system \eqref{p-IPM}.

\begin{lemma}\label{lem:commutator}
	Let $\sigma \in \mathbb{R}$, $1 \leq r \leq \infty$, $1 \leq p \leq p_1 \leq \infty$, and let $v$ be a vector field on $\mathbb{R}^n$. Assume either:
	\begin{align*}
		\sigma > -n \min\left\{\frac{1}{p_1}, \frac{1}{p'}\right\},
	\end{align*}
	or, in the divergence-free case ($\nabla \cdot v = 0$):
	\begin{align*}
		\sigma > -1 - n \min\left\{\frac{1}{p_1}, \frac{1}{p'}\right\},
	\end{align*}
	where $\frac{1}{p} + \frac{1}{p'} = 1$. Define the commutator $R_j := [v \cdot \nabla, \Delta_j]f$. Then there exists $C = C(p,p_1,\sigma,n) > 0$ such that:
	\begin{equation*}
		\left\|(2^{j\sigma}\|R_j\|_{L^p})_j\right\|_{\ell^r} \leq 
		\begin{cases}
			C\|\nabla v\|_{B^{\frac{n}{p}}_{p_1,\infty} \cap L^\infty} \|f\|_{B^\sigma_{p,r}}, & \text{if } \sigma < 1 + \frac{n}{p_1}, \\
			C\|\nabla v\|_{B^{\sigma-1}_{p_1,r}} \|f\|_{B^\sigma_{p,r}}, & \text{if } \sigma > 1 + \frac{n}{p_1} \text{ or } \sigma = 1 + \frac{n}{p_1}, r=1,
		\end{cases}
	\end{equation*}
	where $[v \cdot \nabla, \Delta_j]f = v \cdot \nabla(\Delta_j f) - \Delta_j(v \cdot \nabla f)$.
\end{lemma}

	\end{document}